\documentclass[11pt]{article}

\usepackage{latexsym,mathrsfs}
\usepackage{amsmath,amssymb} 
\usepackage{amsthm,enumerate,verbatim}
\usepackage{amsfonts}
\usepackage{graphicx}
\usepackage{algorithm}
\usepackage{algorithmic}
\usepackage{url}

\renewcommand{\algorithmicrequire}{\textbf{Input:}}
\renewcommand{\algorithmicensure}{\textbf{Output:}}

\newtheorem{theorem}{Theorem}
\newtheorem{corollary}{Corollary}

\newtheorem*{definition*}{Definition}

\newtheorem{remark}{Remark}

\DeclareMathOperator{\argmin}{argmin} 
 
\DeclareMathOperator{\rank}{rank}

\title{Exact and Heuristic Algorithms 
\\ for Semi-Nonnegative Matrix Factorization}
\date{}

\author{Nicolas Gillis \\ 
Department of Mathematics and Operational Research \\ 
Facult\'e Polytechnique, Universit\'e de Mons \\ 
Rue de Houdain 9, 7000 Mons, Belgium\\
 nicolas.gillis@umons.ac.be  
 \and 
  Abhishek Kumar  \\ 
 IBM T.J. Watson Research Center  \\ 
Yorktown Heights, NY 10598  \\ 
abhishk@us.ibm.com
}

\begin{document}

\maketitle

\begin{abstract}
Given a matrix $M$ (not necessarily nonnegative) and a factorization rank $r$, semi-nonnegative matrix factorization (semi-NMF) looks for a matrix $U$ with $r$ columns and a nonnegative matrix $V$ with $r$ rows such that $UV$ is the best possible approximation of $M$ according to some metric.  
In this paper, we study the properties of semi-NMF from which we develop exact and heuristic algorithms. 
Our contribution is threefold. 
First, we prove that the error of a semi-NMF of rank $r$ has to be smaller than the best unconstrained approximation of rank $r-1$. This leads us to a new initialization procedure based on the singular value decomposition (SVD) with a guarantee on the quality of the approximation.  
Second, we propose an exact algorithm (that is, an algorithm that finds an optimal solution), also based on the SVD, for a certain class of matrices (including nonnegative irreducible matrices) from which we derive an initialization for matrices not belonging to that class. 
Numerical experiments illustrate that this second approach performs extremely well, and allows us to compute optimal semi-NMF decompositions in many situations.  
Finally, we analyze the computational complexity of semi-NMF proving its NP-hardness, already in the rank-one case (that is, for $r = 1$), and we show that semi-NMF is sometimes ill-posed (that is, an optimal solution does not exist).  
\end{abstract}

\textbf{Keywords.} semi-nonnegative matrix factorization, semi-nonnegative rank, initialization, algorithms.

\section{Introduction}

Semi-nonnegative matrix factorization (semi-NMF) can be defined as follows: 
Given a matrix $M \in \mathbb{R}^{m \times n}$ and a factorization rank $r$, solve
\begin{equation} \label{semiNMF}
\min_{U \in \mathbb{R}^{m\times r}, V \in \mathbb{R}^{r \times n}} ||M-UV||_F^2 
\quad \text{ such that }\quad
V \geq 0, 
\end{equation}
 where $||.||_F$ is the Frobenius norm and $V \geq 0$ means that $V$ is component-wise nonnegative. 
 Note that any other suitable metric could be used but we focus in this paper on this particular objective function. 
 Semi-NMF has been used in the context of data analysis and clustering 
\cite{DTJ10}. In fact, letting each column of the input matrix represent an element of a data set (there are $n$ elements in dimension~$m$), 
the semi-NMF decomposition can be equivalently written as 
\[
M(:,j) \approx \sum_{k=1}^r U(:,k) V(k,j) \qquad \text{ for all $j$}, 
\]
so that each column of $M$ is a conic combination of the columns of $U$ since $V \geq 0$.  
Each column of $U$ can then be interpreted as a cluster centroid 
while the columns of $V$ are the weights needed to reconstruct approximately each column of $M$ using the columns of $U$ and hence can be interpreted as cluster membership indicators; see the discussion in \cite{DTJ10}. 
Semi-NMF has been used successfully for example for motion segmentation with missing data~\cite{MD12}, image super-resolution~\cite{BRG12} or hyperspectral unmixing~\cite{TCI12}.  

Let us define the semi-nonnegative rank of matrix $M$, denoted $\rank_s(M)$, as the smallest $r$ such that there exists $U \in \mathbb{R}^{m\times r}$ and $V \in \mathbb{R}^{r \times n}$ with $M = UV$ and $V \geq 0$. 
Let us also define \emph{exact semi-NMF}, a problem closely related to semi-NMF, as follows: 
Given a matrix $M \in \mathbb{R}^{m \times n}$, compute its semi-nonnegative rank $r_s = \rank_s(M)$ 
and a corresponding factorization $U \in \mathbb{R}^{m\times r_s}$ and $V \in \mathbb{R}^{r_s \times n}$ such that $M = UV$ and $V \geq 0$.  
We also denote $\rank(M)$ the usual rank of a matrix $M$, and $\rank_+(M)$ its nonnegative rank which is the smallest $k$ such that there exists $U \in \mathbb{R}^{m \times k}_+$ and   $V \in \mathbb{R}^{k \times n}_+$ with $M = UV$; see, e.g., \cite{GG10b} and the references therein. 
By definition, we have 
\[
\rank(M) \leq \rank_s(M) \leq \rank_+(M). 
\] 
The paper is organized as follows. 
\begin{itemize}

\item In Section~\ref{sec2}, we prove that the error of a semi-NMF of rank $r$ has to be smaller than the best unconstrained approximation\footnote{
In the remainder of the paper, unless stated otherwise, we will refer to the best rank-$r$ approximation $X$ of $M$ as an optimal solution of $\min_{X, \rank(X) \leq r} ||M-X||_F^2$. 
Note that the best rank-$r$ unconstrained approximation of a given matrix is not necessarily unique. In fact, it is if and only if the $r$th and $(r+1)$th singular values are distinct; see, e.g., \cite{GV96}. 
However, we will use in this paper this abuse of language as the non-uniqueness issue does not play a role in our developments. 
} 
of rank $r-1$ (Theorem~\ref{th1}), which implies \mbox{$\rank_s(M) \leq \rank(M)+1$}. This leads us to a new initialization procedure for semi-NMF based on the singular value decomposition (SVD) with a guarantee on the quality of the approximation (Algorithm~\ref{svdinit}). 

\item In Section~\ref{sec3}, we prove that solving exact semi-NMF can be done in polynomial time (Theorem~\ref{th3}). 
In particular, we show that $\rank_s(M) = \rank(M)$ if and only if a positive vector belongs to the row space of $M$ (after having removed its zero columns), otherwise $\rank_s(M) = \rank(M)+1$ (Theorem~\ref{th2}). 
We propose an algorithm that solves semi-NMF~\eqref{semiNMF} for a certain class of matrices (which includes nonnegative matrices $M$ for which $M^T M$ is irreducible) and requires one SVD computation (Algorithm~\ref{exactseminmf}). We also generalize this algorithm for matrices not belonging to that class, and, in Section~\ref{sec5}, we show that it performs extremely well, often leading to optimal solutions of semi-NMF~\eqref{semiNMF}. 

\item In Section~\ref{sec4}, we prove that semi-NMF is NP-hard already for $r=1$ (Theorem~\ref{th5}). 
In light of the results above, this shows that computing (approximate) semi-NMF is much more difficult than computing exact semi-NMF (unless $P = NP$). Moreover, we also show that semi-NMF is sometimes ill-posed (that is, an optimal solution does not exist).  

\end{itemize}

\begin{remark}
While finishing up this paper, we noticed the very recent paper \cite{CHB15} (available online 17 October 2014). 
It treats the exact semi-NMF problem, and studies Theorems~\ref{th1} and \ref{th2} of this paper\footnote{Note however that the case of matrices with zero columns is not properly treated in \cite{CHB15}; see Theorem~\ref{th2}.}. 


Our contribution goes further than proving Theorems~\ref{th1} and \ref{th2}, and is rather oriented towards algorithmic aspects:
 we propose 
(i) a polynomial-time algorithm for exact semi-NMF, and 
(ii) a very efficient way to initialize semi-NMF algorithms which is provably optimal for a subclass of matrices; 
see Algorithm~\ref{exactseminmf}.  
Moreover, we prove NP-hardness and ill-posedness of semi-NMF; see Section~\ref{sec4}. 
\end{remark}

\section{Semi-NMF based on Unconstrained Low-Rank Approximations} \label{sec2}

Given any rank-$r$ factorization $(A,B)$ of a matrix $M = AB$, an exact rank-$2r$ semi-NMF can be constructed since 
\[
M 
= AB 
= A (B_+ - B_-) 
= \left[ A, \; -A \right] 
\left[ \begin{array}{c} B_+ \\  B_- \end{array} \right], 
\]
where $B_+ = \max(B,0) \geq 0$ and $B_- = \max(-B,0) \geq 0$ so that $B = B_+ - B_-$. This implies 
\begin{equation} \label{2rank} 
\rank_s(M) \leq 2 \rank(M). 
\end{equation}

\subsection{Tight Upper Bound based on the Usual Rank}

A much better bound than \eqref{2rank} based on the usual rank can be derived. In fact, we now show that any factorization of rank $k$ can be transformed into a semi-NMF of rank $k+1$. 

\begin{theorem}[see also \cite{CHB15}, Lem.~1] \label{th1}
Let $A \in \mathbb{R}^{m \times k}$ and $B \in \mathbb{R}^{k \times n}$. Then, there exists $U \in \mathbb{R}^{m \times (k+1)}$ and $V \in \mathbb{R}^{(k+1) \times n}$ such that $V \geq 0$ and  $UV = AB$. 
\end{theorem}
\begin{proof}
Let us define $\bar{a} = -\sum_i A(:,i) = - Ae$ where $e$ is the vector of all ones of appropriate dimensions. 
Let also 
\[
U = [A \; \bar{a}] 
\qquad 
\text{ and } 
\qquad  
V(:,j) = 
\left( \begin{array}{c} B(:,j) \\ 0 \end{array} \right) 
 + \max\left(0, \max_i(-B_{ij}) \right) e 
\; \in \;  \mathbb{R}^{k+1}_+   , 
\]
for all $1 \leq j \leq n$. We have for all $j$ that 
\begin{align*}
UV(:,j) & = [A \, \bar{a}] \left[ \left( \begin{array}{c} B(:,j) \\ 0 \end{array} \right) 
 + \max\left(0, \max_i(-B_{ij}) \right) e  \right] \\
& = A B(:,j) + \max\left(0, \max_i(-B_{ij}) \right) [A \, \bar{a}]  e = A B(:,j),  
\end{align*}
since $[A \, \bar{a}]  e = 0$ by construction. 
\end{proof}


Theorem~\ref{th1} can be geometrically interpreted as follows: \emph{any set of data points in a $r$-dimensional space can be enclosed in the convex hull of $r+1$ vertices}.  
For example, any set of points in a two-dimensional affine subspace is enclosed in a triangle (the triangle just needs to be big enough to contain all data points); see \cite[Section 4]{CHB15} for more details.

\begin{corollary} \label{cor1}  
For any matrix $M$, we have 
\[
\rank(M) \leq \rank_s(M) \leq \rank(M) + 1. 
\]
\end{corollary}
\begin{proof} 
This follows from Theorem~\ref{th1}. 
\end{proof}

This implies that either $\rank_s(M) = \rank(M)$ or  $\rank_s(M) = \rank(M) + 1$. 
Observe that  

\begin{itemize}
\item The above bound is tight. For example, the matrix
\[
M = \left( \begin{array}{ccc} 1 & 0 & -1 \\ 0 & 1 & -1 \end{array} \right) 
\] 
satisfies $\rank_s(M) = \rank(M) + 1 = 3$ (because the cone spanned by the columns of $M$, namely $\mathbb{R}^2$, 
cannot be represented as a cone spanned by two vectors). 

\item When $n$ is large ($n \gg \rank(M) = r$), in general, $\rank_s(M) = r + 1$. 
In fact, it is not likely for a set of $n$ points in a $r$-dimensional space to be spanned by a cone with $r$ rays when $n \gg r$.  
This would require these vectors to be contained in the same half space (see Section~\ref{sec3} for a complete characterization). 
For example, if we generate these vectors uniformly at random on the unit disk, the probability for these vectors to be in the same half space goes to zero extremely fast as $n$ grows.

\item The function $\rank_s(.)$ is not invariant under transposition, that is, $\rank_s(M)$ is not necessarily equal to $\rank_s(M^T)$, 
although they cannot differ by more than one (see Corollary~\ref{cor1}). 
For example, the matrix
\[
M 
= 
\left( \begin{array}{ccc} -1 & 0 & -1 \\ 0 & -1 & -1 \\ 1 & 1 & 2 \end{array} \right) 
=
\left( \begin{array}{cc} -1 & 0  \\ 0 & -1  \\ 1 & 1 \end{array} \right) 
\left( \begin{array}{ccc} 1 & 0 & 1 \\ 0 & 1 & 1  \end{array} \right) 
\]
satisfies $\rank_s(M) = 2 \neq \rank_s(M^T) = 3$. 

\end{itemize}


\subsection{Algorithm for Semi-NMF} 

A simple yet effective algorithm for semi-NMF is a block coordinate descent method that alternatively optimizes over $U$ for $V$ fixed and over $V$ for $U$ fixed: 
\begin{itemize}
\item  The problem in $U$ is an unconstrained least squares and can be solved with dedicated solvers. 
\item  The problem in $V$ is a nonnegative least squares problem. To solve this problem, we propose to use a block coordinate descent method on the rows of $V$ since the optimal solution for a given row (all other rows being fixed) has a closed-form solution; 
see, e.g., \cite{GG12} and the references therein.  
\end{itemize}
Algorithm~\ref{cdsemi} implements this strategy 
and is guaranteed to converge to a stationary point of \eqref{semiNMF} 
because each block of variables is optimized exactly and achieves a unique global minimizer\footnote{Given that $U$ and $V$ remain full rank.}~\cite{B99, B99b} (Prop.~2.7.1). (Note that a value of maxiter between 100 and 500 usually gives good results, although this depends on the initialization and the dimensions $m$, $n$ and $r$; see Section~\ref{sec5} for some numerical experiments.)
\algsetup{indent=2em}
\begin{algorithm}[ht!]
\caption{Coordinate Descent for Semi-NMF \label{cdsemi}}
\begin{algorithmic}[1] 
\REQUIRE A matrix $M \in \mathbb{R}^{m \times n}$, an initialization $V \in \mathbb{R}^{r \times n}_+$, a maximum number of iterations maxiter. 
\ENSURE A rank-$r$ semi-NMF $(U,V)$ of $M \approx UV$ with $V \geq 0$. 
    \medskip  

\FOR{$i$ = 1 : maxiter}
	\STATE  $U \leftarrow \argmin_{X \in \mathbb{R}^{m \times r}} ||M - XV||_F^2$ \quad ($= M/V$ in Matlab) 
	\STATE \emph{\% Coordinate descent on the rows of $V$}
	\FOR{$i$ = 1 : $r$} 
		\STATE 
		\begin{align*} V(i,:)^T 	
		& \leftarrow \argmin_{x \in \mathbb{R}^{n}_+} ||M - U(:,\mathcal{I})V(\mathcal{I},:) - U(:,i) x^T||_F^2 \\ 
		& = \max\left(0, \frac{ \left( M - U(:,\mathcal{I})V(\mathcal{I},:) \right)^T U(:,i) }{||U(:,i)||_2^2} \right), \mathcal{I} = \{1,\dots,r\} \backslash \{i\}. 
		\end{align*}
\ENDFOR
\ENDFOR
\end{algorithmic}  
\end{algorithm} 

\begin{remark}[Original Semi-NMF Algorithm] \label{algopami}
In the original paper introducing semi-NMF~\cite{DTJ10}, the proposed algorithm is the following: 
\begin{itemize}  
\item the matrices $U$ and $V$ are initialized using k-means: the columns of $U$ are taken as the cluster centroids of the columns of $M$, while $V$ is the binary indicator matrix to which the constant 0.2 is added 
(for the multiplicative updates to be able to modify all entries of $V$; see below). 

\item $V$ is updated using the following multiplicative updates: for all $k,j$, 
\[
V_{kj} 
\leftarrow 
V_{kj} 
\; 
\sqrt{ \frac{  \max(0,(U^T M)_{kj}) + \max(0,-(U^T U V)_{kj}) }
{ \max(0,-(U^T M)_{kj}) + \max(0,(U^T U V)_{kj}) } } . 
\]
These updates are guaranteed to decrease the objective function. 

\item $U$ is updated as in Algorithm~\ref{cdsemi}, using the optimal solution for $V$ fixed. 

\end{itemize}  
Hence this algorithm is rather similar to Algorithm~\ref{cdsemi}, where $V$ would be initialized with k-means and would be updated with the multiplicative updates. 
However, compared to Algorithm~\ref{cdsemi}, the algorithm from~\cite{DTJ10} suffers from the following drawbacks: 
\begin{itemize}  
\item It is not guaranteed to converge to a stationary point (nonincreasningness is not a sufficient condition). 

\item It has a locking phenomenon: once an entry of matrix $V$ is set to zero, it cannot be modified (because of the multiplicative nature). 

\item It sometimes runs into numerical problems, because the denominator in the update rules is equal to zero.

\item Although it has almost exactly the same computational cost as Algorithm~\ref{cdsemi} 
(the update of $V$ requires the matrix products $U^T M$ and $U^T U V$ in both cases), 
it converges significantly slower. The same observation was made by several works comparing coordinate descent approaches to multiplicative updates for optimizing $U$ and $V$ in NMF~\cite{CP09b, LZ09, HD11, LW12, GG12}. 

\end{itemize} 
Moreover, in this paper, our goal is not to compare strategies to update matrices $U$ and $V$ but rather to compare initialization strategies. For these reasons, we do not use the algorithm from~\cite{DTJ10} in this paper, but we will compare their initialization based on $k$-means to our proposed approaches; see Section~\ref{sec5}.    
\end{remark}

\subsection{SVD-based Initialization} \label{sec23}

We can use the construction of Theorem~\ref{th1} to initialize semi-NMF algorithms such as Algorithm~\ref{cdsemi}; 
see Algorithm~\ref{svdinit}.  
Given a rank-$r$ approximation $(A,B)$ of $M \approx AB$ computed via the truncated SVD, we flip the sign of the rows of $B$ (and the columns of $A$ accordingly) so that the minimum on each row of $B$ is maximized. (Note that other sign permutations exist; see, e.g.,~\cite{BAK08}.) 
The motivation behind this choice is to reduce the effect of the correction done at step 5 of Algorithm~\ref{svdinit}. 
We have the following result: 
\begin{corollary} Let $M \in  \mathbb{R}^{m \times n}$, and let $M_{r-1}$ be its best rank-$(r-1)$ approximation with respect to the norm $||.||$,  
then 
\begin{equation} \label{ubopt}
\min_{U \in \mathbb{R}^{m \times r}, V \in \mathbb{R}^{r \times n}_+} ||M - UV|| 
\leq 
||M - M_{r-1}|| .
\end{equation}
The solution provided by Algorithm~\ref{cdsemi} initialized with Algorithm~\ref{svdinit} satisfies this bound for the Frobenius norm. 
\end{corollary}
\begin{proof} 
This follows directly from Theorem~\ref{th1}, and the fact that Algorithm~\ref{cdsemi} generates a sequence of iterates monotonically decreasing the objective function. 
\end{proof} 

\algsetup{indent=2em}
\begin{algorithm}[ht!]
\caption{SVD-based Initialization for Semi-NMF (see Theorem~\ref{th1}) \label{svdinit}}
\begin{algorithmic}[1] 
\REQUIRE A matrix $M \in \mathbb{R}^{m \times n}$ and a factorization rank $r$. 
\ENSURE A rank-$r$ semi-NMF $(U,V)$ of $M \approx UV$ with $V \geq 0$ achieving the same error than the best rank-$(r-1)$ approximation of $M$. 
    \medskip   
\STATE $[A,S,B^T] =$ svds$(M,r-1)$ ; \% \emph{See the Matlab function \texttt{svds} }
\STATE $A = AS$;  
\STATE For each $1 \leq i \leq r-1$: multiply $B(i,:)$ and $A(:,i)$ by $-1$ if $\min_j B(i,j) \leq \min_j (-B(i,j))$ ; 
\STATE $U = \left[A \; \; -Ae\right]$; 
\STATE $V(:,j) =  \left( \begin{array}{c} B(:,j) \\ 0 \end{array} \right) 
 + \max\left(0, \max_i(-B_{ij}) \right) e 
\qquad 1 \leq j \leq n$. 
\end{algorithmic}  
\end{algorithm}

In Section~\ref{sec5}, we will compare Algorithm~\ref{svdinit} with several other initialization strategies. It turns out that, although it is an appealing solution from a theoretical point of view
(as it guarantees a solution with error equal to the error of the best rank-$(r-1)$ approximation of $M$), 
it performs relatively poorly, in most cases worse than random initializations. 

\section{Exact Algorithm for Semi-NMF} \label{sec3}

In the previous section, we showed that for any matrix $M$, $\rank_s(M)$ equals $\rank(M)$ or $\rank(M)+1$. 
In this section, we first completely characterize these two cases, and derive an algorithm to solve the exact semi-NMF problem.  
\begin{theorem}[\footnote{This result is very similar to \cite[Th.3]{CHB15}. However, the case of matrices with zero columns is not treated properly in \cite[Th.3]{CHB15}.  For example, according to \cite[Th.3]{CHB15}, $\rank(0) \neq \rank_s(0)$ which is incorrect. 
In fact, the authors claim that: 
`As a consequence, $B$ contains a zero column which contradicts the fact that $\rank(B) = r$'. This is not true: if $n > r$ (which is usually the case since $r = \rank(M) \leq n$), $B$ can contain zero columns while $\rank(B) = r$.}] \label{th2}
Let $M \in \mathbb{R}^{m \times n}$. The following statements are equivalent
\begin{enumerate}
\item[(i)] $\rank(M) = \rank_s(M)$. 

\item[(ii)] There exists a non-zero vector $z \in \mathbb{R}^m$ such that $M(:,j)^T z > 0$ for all $j$ such that $M(:,j) \neq 0$. 
In other terms, all non-zero columns of $M$ belong to the interior of a half space $\mathcal{P}_z = \{ x \in \mathbb{R}^m \ | \ x^T z \geq 0 \}$ for some $z \neq 0$, or, equivalently, there exists a positive vector 
in the rows space of $M$  after its zero columns have been removed.  


\item[(iii)] Given any factorization $(A,B) \in \mathbb{R}^{m \times r} \times \mathbb{R}^{r \times n}$ of $M = AB$ with $r = \rank(M)$, 
all non-zero columns of $B$ belong to the interior of a half space $\mathcal{P}_y$ for some $y \neq 0$. 

\end{enumerate}
\end{theorem} 
\begin{proof} 
We assume without loss of generality (w.l.o.g.) that $M$ does not contain a zero column (otherwise discard it, which does not influence the conditions above).

The equivalence $(ii) \iff (iii)$ follows from simple linear algebra. 
Since the columns of the matrix $M$ belong to the interior of $\mathcal{P}_z$, we have $M^T z > 0$. 
Without loss of generality, we can take $z = Aw$ for some $w$. In fact, let us denote $A^{\bot}$ the orthogonal complement of $A$ so that, for any $z$, there exists $w$ and $w^{\bot}$ with $z = Aw + A^{\bot} w^{\bot}$ for which we have 
\[
0 < M^T z = B^T A^T \left( Aw + A^{\bot} w^{\bot} \right) = B^T A^T (Aw) = M^T (Aw) . 
\]
Hence, replacing $z$ with $Aw$ does not modify $M^T z > 0$. 
Moreover, the derivation above shows that the columns of $B$ belong to the half space $\mathcal{P}_y$ with $y = A^T Aw = A^T z$ which proves $(ii) \Rightarrow (iii)$. Proving the direction $(iii) \Rightarrow (ii)$ is similar: the left inverse $A^{\dagger} \in \mathbb{R}^{r \times m}$ of $A$ exists since $\rank(A)$ must be equal to $r$ ($A^{\dagger} A = I_r$) and taking $z = {A^{\dagger}}^T y$, we have 
\[
0 < B^T y 
= B^T \left(A^{\dagger} A\right)^T y
= B^T A^T {A^{\dagger}}^T y = M^T z. 
\]

Let us show $(iii) \Rightarrow (i)$. We have $y \in \mathbb{R}^r$ such that $x = B^T y > 0$. 
Note that the rows of $B$ are different from zero since $B \in \mathbb{R}^{r \times n}$ and $\rank(B) = r$. 
 Hence we can flip the sign of the rows of $B$ along with the corresponding entries of $y$ (keeping $B^Ty$ unchanged) so that the maximum entry on each row of $B$ is positive, that is, $\max_j B(i,j) > 0$ for all~$i$ (see for example step~2 of Algorithm~\ref{exactseminmf}). 

Let  
\[
V(i,:) = B(i,:) +  \alpha_i x^T = B(i,:) +  \alpha_i y^T B = (e_i + \alpha_i  y)^T B, 
\]
where $\alpha_i = \max\left(0, \max_j \frac{-B(i,j)}{x_j}\right)$
 for all $1 \leq i \leq r$ so that $V(i,:) \geq 0$, and $e_i$ is the $i$th column of the identity matrix. In other terms, 
\[
V = B + \alpha x^T = B +  \alpha y^T B  = \left(I + \alpha y^T \right) B \geq 0. 
\]
We can take $U = A (I +  \alpha  y^T )^{-1}$ so that $M = A B = UV$ with $V \geq 0$.  
Using Sherman-Morrison formula, we have that 
\[
\left(I +  \alpha  y^T \right)^{-1} = I - \frac{\alpha y^T}{ 1  + y^T \alpha }, 
\] 
hence $U$ can be computed given that $y^T \alpha \neq - 1$. It remains to show that $y^T \alpha \neq - 1$. We have 
\begin{align*}
y^T \alpha = \sum_{i=1}^r y_i \alpha_i 
& = \sum_{i=1}^r y_i \max\left(0, \max_j \frac{-B(i,j)}{x_j}\right) \\
& > \sum_{i=1}^r y_i \left( \frac{1}{n} \sum_{j=1}^n \frac{-B(i,j)}{x_j}\right) \\
& = \frac{-1}{n} \sum_{j=1}^n  \left(   \frac{ \sum_{i=1}^r  B(i,j) y_i}{ \sum_{k=1}^r B(k,j) y_k }\right) = - 1. 
\end{align*}
The strict inequality follows from the fact that $\max_j B(i,j) > 0$ for all~$i$. 

Let us show $(i) \Rightarrow (ii)$. Let $r = \rank(M) = \rank_s(M)$, and $M = UV$ with $U \in \mathbb{R}^{m \times r}$ and $V \in \mathbb{R}^{r \times n}_+$ where $\rank(U) = \rank(V) = r$. Since no column of $M$ is equal to zero, 
no column of $V$ is hence $V^T e > 0$. Since $U$ is full rank, its left inverse $U^{\dagger} \in \mathbb{R}^{r \times m}$ exists ($U^{\dagger} U = I_r$). This implies that 
\[
 M^T \left({U^{\dagger}}^T e \right) = V^T U^T {U^{\dagger}}^T e = V^T (U^{\dagger} U)^T e  = V^T e > 0. 
\] 
\end{proof}

In the remainder of the paper, we say that \emph{a matrix $M$ is semi-nonnegative} if and only if $\rank(M) = \rank_s(M)$ if and only if the non-zero columns of $M$ are contained in the interior of a half space.

\begin{remark} \label{srkn}
Note that if $\rank(M) = n$, then $M \in \mathbb{R}^{m \times n}$ necessarily contains a positive vector in its row space 
(since it spans $\mathbb{R}^n$) hence $\rank_s(M) = \rank(M) = n$. 
We also have $\rank_s(M) \leq n$ using the trivial decomposition $M = M I_n$ where $I_n$ is the $n$-by-$n$ identity matrix; see also \cite[Lemma~1]{CHB15}. 
\end{remark}

\begin{theorem} \label{th3}
Given a matrix $M$, 
solving exact semi-NMF can be done in polynomial time (both in the Turing machine model and the real model of computation).  
\end{theorem}
\begin{proof} 
By Theorem~\ref{th2}, it suffices to check whether a positive vector belongs to the row space of $M$ (after having discarded the zero columns). For example, one can check whether the following linear system of inequalities has a solution: 
\begin{equation} \label{sysineq} 
 M(:,j)^T z \geq 1 \text{ for all $j$ such that $M(:,j) \neq 0$}. 
\end{equation}
If the system is feasible 
(which can be checked in polynomial time, both in the Turing machine model and the real model of computation \cite{B04}), 
$\rank_s(M) = \rank(M)$, otherwise $\rank_s(M) = \rank(M)+1$. 
The rank of a matrix and a corresponding low-rank factorization can be computed in polynomial time as well, e.g., using row-echelon form \cite{E67}. 
The factorization can be transformated into an exact semi-NMF using the construction of Theorem~\ref{th1} in the case $\rank_s(M) = \rank(M)+1$ and of Theorem~\ref{th2} in the case $\rank_s(M) = \rank(M)$. 
\end{proof} 

In practice, it is better to compute a factorization of $M$ using the singular value decomposition, which is implemented in our algorithms 
(Algorithms~\ref{svdinit} and~\ref{exactseminmf}). 

We have just showed how to compute an exact NMF. The same result can actually be used to compute approximate semi-NMF, 
given that the columns of the best rank-$r$ approximation of $M$ are contained in the same half space. 
\begin{corollary} \label{corsemiNMF}
Let $M \in \mathbb{R}^{m \times n}$. If the rank-$r$ truncated SVD of $M$ is semi-nonnegative, then 
semi-NMF~\eqref{semiNMF} can be solved in polynomial time in $m$, $n$ and $\mathcal{O}(\log(1/\epsilon))$ 
where $\epsilon$ is the precision of the truncated SVD decomposition.  
Algorithm~\ref{exactseminmf} is such a polynomial-time algorithm. 
\end{corollary}
\begin{proof} 
This follows from Theorem~\ref{th2} and the fact that the rank-$r$ truncated SVD provides an optimal rank-$r$ approximation and can be computed up to any precision $\epsilon$ in time polynomial  in $m$, $n$ and $\mathcal{O}(\log(1/\epsilon))$; see, e.g., \cite{TB97, V91} and the references therein.    
\end{proof}

If no positive vector belongs to the row space of the second factor $B$ of a rank-$r$ factorization $AB$, then, by Theorem~\ref{th2}, 
there does not exist a rank-$r$ semi-NMF $U$ and $V\geq 0$ such that $AB = UV$ and the linear system~\eqref{sysineq} is not feasible.  
In that case, we propose to use the following heuristic: solve 
\begin{equation} \label{yepsi}
\min_{y \in \mathbb{R}^r, \epsilon \in \mathbb{R}_+}  \epsilon  
 \quad 
\text{ such that } 
\quad 
(B(:,j)+ \epsilon \, e)^T y \geq 1 \text{ for all $j$ such that $B(:,j) + \epsilon \, e \neq 0$}. 
\end{equation}
Although Problem~\eqref{yepsi} is not convex, a solution can be obtained using a bisection method on the variable $\epsilon$. 
In fact, the optimal solution $\epsilon^*$ will belong to the interval $[0, \epsilon_+]$ where 
$\epsilon_+ = \max_{i,k} \max(-B_{ik},0)$ since $B + \epsilon_+ \geq 0$ hence the problem is feasible (e.g., $y = e$). 
Note that the bisection method first checks whether $\epsilon = 0$ is feasible in which case it terminates in one step and returns an optimal semi-NMF. In our implementation, we used a relative precision of $10^{-3}$, that is, we stop the algorithm as soon as $\epsilon_f - \epsilon_i \leq 10^{-3} \epsilon_+$, where 
$\epsilon_f$ is the smallest feasible $\epsilon$ found so far 
(initialized at $\epsilon_+$), and 
$\epsilon_i$ is the largest infeasible $\epsilon$ found so far (initialized at 0) so that our bisection procedure has to solve at most ten linear systems (since $0.001 > 2^{-10}$). The reason we choose a relatively low precision is that high precision is not necessary because, when the optimal $\epsilon^* \neq 0$, 
the algorithm will be used as an initialization procedure for Algorithm~\ref{cdsemi} that will refine the semi-NMF solution locally.

Algorithm~\ref{exactseminmf} implements this strategy and will be used in Section~\ref{sec5} to initialize Algorithm~\ref{cdsemi} and will be shown to perform extremely well. 

\algsetup{indent=2em}
\begin{algorithm}[ht!]
\caption{Heuristic for Semi-NMF \label{exactseminmf} (see Theorem~\ref{th2} and Corollary~\ref{corsemiNMF})} 
\begin{algorithmic}[1] 
\REQUIRE A matrix $M \in \mathbb{R}^{m \times n}$, a factorization rank $r$. 
\ENSURE A rank-$r$ semi-NMF $(U,V)$ of $M \approx UV$ with $V \geq 0$. 
    \medskip  
		
\STATE $[A,S,B^T] =$ svds$(M,r)$ ; \% \emph{See the Matlab function \texttt{svds} } 

\STATE For each $1 \leq i \leq r$: multiply $B(i,:)$ by $-1$ if $\min_j B(i,j) \leq \min_j (-B(i,j))$ ; 

\STATE Let $(y^*, \epsilon^*)$  be the optimal solution of the following the optimization problem 
\[ 
\min_{y \in \mathbb{R}^r, \epsilon \in \mathbb{R}_+}  \epsilon  
 \quad 
\text{ such that } 
\quad 
(B(:,j)+ \epsilon \, e)^T y \geq 1 \text{ for all $j$ such that $B(:,j) + \epsilon \, e \neq 0$}. 
\] 
\emph{\% If $\epsilon^* = 0$ ($\iff$ $B$ is semi-nonnegative), then the heuristic is optimal. }

\STATE $x = (B + \epsilon^* \, 1_{r \times n})^T y^* \geq 1$ ;  \quad \emph{\% $1_{r \times n}$ is the $r$-by-$n$ matrix of all ones. }
\STATE $\alpha_i = \max\left(0, \max_{j} 
\frac{-B(i,j)}{x(j)} \right)$ for all $1 \leq i \leq r$ ; 
\STATE $V = B + \alpha x^T $ ; 

\STATE  $U \leftarrow \argmin_{X \in \mathbb{R}^{m \times r}} ||M - XV||_F^2$ \quad ($= M/V$ in Matlab). 
\end{algorithmic}  
\end{algorithm}

\begin{remark}\label{remeps}
It is interesting to note that the value of $\epsilon^*$ tells us how far $B$ is from being semi-nonnegative (hence $AB$; see Theorem~\ref{th2}). 
In fact, by construction, the matrix $B_{\epsilon^*} = B + \epsilon^* 1_{r \times n}$ is semi-nonnegative. The idea behind Algorithm~\ref{exactseminmf} is to replace $B$ with its semi-nonnegative approximation $B_{\epsilon^*}$. 
If $B$ is close to being semi-nonnegative, $\epsilon^*$ will be small and Algorithm~\ref{exactseminmf} will perform well; see Section~\ref{sec5} for the numerical experiments. 
Note that other strategies for finding a semi-nonnegative matrix close to $B$ are possible and it would be interesting to compare them with Algorithm~\ref{exactseminmf}: this is a direction for further research. 
\end{remark}

\subsection{Nonnegative Matrices} \label{secnm}

Theorem~\ref{th2} implies that 
\begin{corollary} 
Let $M \in \mathbb{R}^{m \times n}_+$, then $\rank_s(M) = \rank(M)$. 
\end{corollary} 
\begin{proof}
In fact, any nonnegative vector different from zero belongs to the interior of the half space $\mathcal{P}_e = \{ x \in \mathbb{R}^m | \sum_{i=1}^m x_i \geq 0 \}$. 
\end{proof} 

We have seen that if the best rank-$r$ approximation of a matrix contains a positive vector in its row space, then an optimal semi-NMF of the corresponding matrix can be computed; see Corollary~\ref{corsemiNMF}. 
This will be in general the case for nonnegative matrices. In fact, the Perron-Frobenius theorem guarantees that this will be the case when $M^T M$ is a irreducible nonnegative matrix (since its first eigenvector can be chosen positive). Recall that a matrix $A$ is irreducible if the graph induced by $A$ is strongly connected (every vertex is reachable from every other vertex).  
\begin{corollary}  \label{irreduc}
Let $M \in \mathbb{R}^{m \times n}_+$. 
If $M^T M$ is irreducible, then semi-NMF~\eqref{semiNMF} can be solved via the truncated SVD for any rank $r$. 
\end{corollary} 

Corollaries~\ref{corsemiNMF} and~\ref{irreduc}  suggest that \emph{in almost all cases} semi-NMF of nonnegative matrices can be computed using a simple transformation of an unconstrained approximation (such as the truncated SVD).
This observation challenges the meaning of semi-NMF of nonnegative matrices: does semi-NMF of nonnegative matrices really make sense? 
In fact, most nonnegative matrices encountered in practice are irreducible, and, even if they are not, 
it is likely for Corollary~\ref{corsemiNMF} to hold since the columns of the best rank-$r$ approximation of a nonnegative matrix are likely to be close to the nonnegative orthant hence belong to a half space (in particular $\mathcal{P}_e$).  (Note that, by the Perron-Frobenius theorem, for $M \in \mathbb{R}^{m \times n}_+$ and $r=1$, there always exists a nonnegative best rank-1 approximation.)  
In these cases, semi-NMF can be solved by a simple transformation of the SVD and it is not clear what semi-NMF brings to the table. 

Note that several authors have proposed semi-NMF algorithms and applied them to nonnegative matrices, e.g., multiplicative updates where proposed in \cite{DTJ10, TBZS14}. 
Our results shows that, from a theoretical point of view, this does not really make sense (since local optimization techniques such as the multiplicative updates usually converge relatively slowly and are not guaranteed to converge to an optimal solution). 
However, it is interesting to check whether the solutions obtained with these heuristics (always) generate optimal solutions under the above conditions. We will see in Section~\ref{sec5} that Algorithm~\ref{cdsemi} does not always converge to an optimal solution for (semi-)nonnegative matrices for all initializations (in particular, when $r$ is large). 

A direction for further research that would make sense for semi-NMF of nonnegative matrices is to add structure to the factors $U$ and/or $V$. For example, imposing $V$ to be sparse would enhance the clustering property of semi-NMF. (Note that the construction of Theorem~\ref{th2} usually generates a matrix $V$ with a single zero per row.) In fact, if $V$ is required to have a single non-zero entry per column equal to one, semi-NMF reduces to $k$-means~\cite{DTJ10}.

\begin{remark}
The results of this section also apply to nonpositive matrices since $M$ is nonpositive if and only if $-M$ is nonnegative. Hence if we have a semi-NMF of $-M = UV$, we have a semi-NMF for $M = (-U) V$. 
\end{remark}

\subsection{Semi-Nonnegative Matrices} \label{snm}

If one performs a semi-NMF of a semi-nonnegative matrix $M$ with factorization rank $r = \rank(M) = \rank_s(M)$, then, by Theorem~\ref{th2}, the solution computed by Algorithm~\ref{exactseminmf} will be optimal. 
However, if $r < \rank(M)$, it is not guaranteed to be the case. 
In this section, we provide a sufficient condition for Algorithm~\ref{exactseminmf} to be optimal for semi-nonnegative matrices $M$ when $r < \rank_s(M)$; 
see Theorem~\ref{th4}.  
Intuitively, the idea is the following: the columns of the best rank-$r$ approximation $X$ of $M$ should be relatively close to the columns of $M$ hence it is likely that they also belong to a half space. In that case, by Corollary~\ref{corsemiNMF}, Algorithm~\ref{exactseminmf} is optimal.  
Note that if the best rank-$k$ approximation of $M$ contains a positive vector in its row space, then the best rank-$r$ approximation of $M$
 for all $r \geq k$ does as well since optimal low-rank approximations can be computed one rank-one factor at a time; see, e.g., \cite{GV96}. 



\begin{theorem} \label{th4} 
Let $M$ be a semi-nonnegative matrix so that there exist $z$ with $M(:,j)^Tz > 0$ for all $j$ such that $M(:,j) \neq 0$ and $||z||_2 = 1$. 
Let $X$ be an approximation of $M$ such that 
\[
||M(:,j) - X(:,j)||_2 
\quad < \quad 
M(:,j)^Tz \quad \text{ for all $j$ such that $M(:,j) \neq 0$}, 
\]
and $X(:,j) = 0$ whenever $M(:,j) = 0$ (which is optimal and does not influence the rank of $X$). 
Then $X$ is semi-nonnegative, that is, there exists a rank-$r$ semi-NMF $(U,V)$ such that $X = UV$. 
\end{theorem}
\begin{proof} 
Let us denote the residual of the approximation $E = M - X$. 
For all $j$ such that $M(:,j) \neq 0$ we have  
\[
X(:,j)^T z = \left( M(:,j) - E(:,j) \right)^T z
= M(:,j)^T z -  E(:,j)^T z \geq M(:,j)^T z - ||E(:,j)||_2 > 0,  
\]
while $M(:,j) = 0$ implies $X(:,j) = 0$, hence the result follows from Theorem~\ref{th2}. 
\end{proof}


Note that we have for all $j$ that 
\mbox{$\lVert M(:,j) - X(:,j)\rVert_2 \leq \sigma_{k+1}$} for $X$ being the best rank-$k$ approximation of $M$,  where $\sigma_{k+1}$ is the $(k+1)$th singular value of $M$. 
Hence the smaller  
$\sigma_{k+1}$ is, the more likely it is for $X$ to be semi-nonnegative. 
This also means that the larger $k$ is, 
the more likely it is for Algorithm~\ref{exactseminmf} to perform well (in particular, to return an optimal semi-NMF). 
This will be illustrated in Section~\ref{sec5} with some numerical experiments.

\section{Computational Complexity and Ill-posedness of Semi-NMF} \label{sec4}

Despite the positive results described in the previous sections, 
the semi-NMF problem in the general case 
(that is, when the input matrix is not close to being semi-nonnegative) seems more difficult. 
The rank-one semi-NMF problem is the following: given $M \in \mathbb{R}^{m \times n}$, solve 
\begin{equation} \label{semi1a}
\min_{u \in \mathbb{R}^m, v \in \mathbb{R}^n} ||M-uv^T||_F^2 \quad \text{ such that } v \geq 0. 
\end{equation}

\begin{theorem} \label{th5}
Rank-one semi-NMF \eqref{semi1a} is NP-hard. 
\end{theorem}
\begin{proof}
Assume w.l.o.g.\@ that $||v||_2 = 1$. Then the optimal solution for $u$ is given by $u^* = Mv$. 
Therefore, at optimality,  
\begin{align*}
||M-uv||_F^2 
& = ||M||_F^2 - 2 u^T M v + ||uv^T||_F^2  \\
& = ||M||_F^2 - 2 v^T M^T M v + ||Mv||_2^2 ||v||_2^2 \\
& = ||M||_F^2 - 2 ||Mv||_2^2 + ||Mv||_2^2  \\
& =  ||M||_F^2 - ||Mv||_2^2 = ||M||_F^2 - v^T (M^T M) v ,  
\end{align*}
hence \eqref{semi1a} is equivalent to 
\begin{equation} \label{semi1b}
\max_{v \in \mathbb{R}^n} v^T (M^T M) v \quad \text{ such that } \quad  v \geq 0 \text{ and } ||v||_2 = 1. 
\end{equation} 
Since $M$ is arbitrary,  $M^T M$ can represent any semidefinite positive matrix hence rank-one semi-NMF is equivalent to maximizing a convex quadratic over the unit ball in the nonnegative orthant. 
As explained by Noah D.\@ Stein on Mathoverflow.net\footnote{See \url{http://mathoverflow.net/questions/48843/non-negative-quadratic-maximization}}, the problem~\eqref{semi1b} is equivalent to 
\begin{equation} \label{semi1c} 
\max_{v \in \mathbb{R}^n} v^T B v \quad \text{ such that } \quad  v \geq 0 \text{ and } ||v||_2 = 1, 
\end{equation} 
where $B$ is any symmetric matrix (not necessarily semidefinite positive). In fact, if $B$ is not semidefinite positive,  
one can consider the problem with $B - \lambda_{\min}(B) I_n \succeq 0$,  
where $A \succeq 0$ indicates that the matrix $A$ is positive semidefinite. 
In fact, it only changes the objective function by a constant value since, 
for $||v||_2 = 1$, 
\[  
v^T \left(B - \lambda_{\min}(B) I_n \right) v = v^T B v - \lambda_{\min}(B). 
\] 
Let us use the following result: Checking copositivity of a symmetric matrix $C$, that is, checking whether the optimal value of  
\[
\min_{v \geq 0, ||v||_2 = 1} v^T C v 
\]
is nonnegative, is co-NP-complete~\cite{MK87} 
(a decision problem is co-NP-complete if it is a member of co-NP--its complement is in NP--, and any problem in co-NP can be reduced to it in polynomial time; see \cite{Ar09} for more details).     
Since this problem can be solved using \eqref{semi1c} with $B = -C$, 
this implies that rank-one semi-NMF~\eqref{semi1a} is NP-hard.  
\end{proof}

In comparison to NMF, this is a bit surprising: In fact, rank-one NMF can be solved in polynomial time (this follows from the Perron-Frobenius and Eckart-Young theorems) although it is NP-hard in general~\cite{V09}. 
The reason behind this difference is that semi-NMF allows both negative and positive elements in the input matrix (clearly, rank-one semi-NMF of nonnegative matrices can also be solved in polynomial time; see Section~\ref{secnm}). In fact, rank-one NMF is also NP-hard if the input matrix is allowed to have both positive and negative signs \cite[Cor.1]{GG14}, 
that is, given a matrix $M \in \mathbb{R}^{m \times n}$, the problem 
\[
\min_{u \in \mathbb{R}^m, v \in \mathbb{R}^n} ||M-uv^T||_F^2 
\quad \text{ such that } \quad 
 u \geq 0 \text{ and } v \geq 0. 
\]
is NP-hard. \\ 

Let us now show that semi-NMF is not always a well-posed problem (this question was raised by Mohamed Hanafi in a personal communication), 
that is, that an optimal solution of~\eqref{semiNMF} does not always exist. 
Here is a simple example: 
\[ 
M =  \left( \begin{array}{ccc} 
1   & -1  &  0 \\
0   & 0  &   1 \\ 
\end{array} \right) . 
\]
The columns of $M$ belong to the same two-dimensional half space $\{ x \in \mathbb{R}^2 \ | \ x_2 \geq 0 \}$. 
However, the first two columns are on the boundary of that half space. 
Therefore, they are not contained in its interior hence $\rank_s(M) = 3$ (Theorem~\ref{th2}). 
However, the infimum of~\eqref{semiNMF} for $r = 2$ is equal to zero, taking
\[
U =  \left( \begin{array}{cc} 
1   & -1   \\
\delta   & \delta \\ 
\end{array} \right)  \; \text{ and } \; 
V =  \left( \begin{array}{ccc} 
1   & -1  & ( 2 \delta )^{-1} \\
0   & 0  & ( 2 \delta )^{-1} 
\end{array} \right) , 
\; \text{ with  } \;  
UV = 
\left( \begin{array}{ccc} 
1   & -1  &  0 \\
\delta   & \delta  &   1 \\ 
\end{array} 
\right) 
\]
and making $\delta$ tend to zero. 

Note that it is not likely for semi-NMF problems to be ill-posed, 
this only happens when the best cone approximating the columns of $M$ does not exist as it should be a half space.

\section{Numerical Experiments}  \label{sec5}

In  this section, we compare four strategies to initialize Algorithm~\ref{cdsemi} (which only requires the matrix $V$ as an input): 
\begin{enumerate}

\item \emph{Random initialization (RD)}: each entry of $V$ is generated following the uniform distribution in the interval $[0,1]$, that is, \texttt{V = rand(r,n)} in Matlab notations (which we will reuse in the following). 

\item \emph{K-means (KM)}: $V$ is taken as the binary cluster indicator matrix generated by k-means 
($V_{kj} = 1$ if and only if the $j$th column of $M$ belongs to the $k$th cluster) to which is added the constant\footnote{A priori, because we use Algorithm~\ref{cdsemi} and not the multiplicative updates of~\cite{DTJ10}, we do not have to add a constant to $V$. 
However, we observed that it allows Algorithm~\ref{cdsemi} to converge faster, and to better stationary points. 
In fact, taking $V$ as the binary cluster indicator matrix seems to induce some kind of locking phenomenon as Algorithm~\ref{cdsemi} has difficulties to get away from this initial point.} 0.2. This is the initialization from~\cite{DTJ10}, although we do not use their algorithm to update $U$ and $V$ because Algorithm~\ref{cdsemi} is numerically more stable and has much better convergence properties; see Remark~\ref{algopami}.

\item \emph{Algorithm~\ref{svdinit} (A2)}: we use Algorithm~\ref{svdinit} to initialize $V$. 
This initialization guarantees the error to be same as the error of the best rank-$(r-1)$ approximation. 

\item \emph{Algorithm~\ref{exactseminmf} (A3)}: we use Algorithm~\ref{exactseminmf} to initialize $V$. This initialization generates an optimal solution for matrices whose best rank-$r$ approximation contains a positive vector in its row space 
(which will be the case for example when $M$ is nonnegative and $M^TM$  irreducible). 

\end{enumerate} 
In the following two subsections, we generate several synthetic data sets where the dimensions of the input matrix $M$ are $m = 100$ and $n = 200$, and the factorization rank is $r = 20$ and $r = 80$. 
For each generated synthetic matrix, we run Algorithm~\ref{cdsemi} with the four different initializations and consider the error obtained after 10 and 100 iterations. 
We use the notation RD/10 (resp.\@  RD/100) 
to refer to the algorithm that performs 10 (resp.\@ 100) iterations of Algorithm~\ref{cdsemi} using RD as an initialization. 
We use the same notation for the three other initializations KM, A1 and A2, namely KM/10 and KM/100, A2/10 and A2/100, and A3/10 and A3/100. 
We also test the different initialization strategies on real data in Section~\ref{realdata}.  


In order to compare meaningfully the error of solutions obtained for different input matrices, we use the following measure: 
given a semi-NMF $(U,V)$ of $M$, 
\begin{equation} \label{normrelerr}
\text{quality}(U,V) \quad = \quad 100 \left( \frac{||M-UV||_F}{||M-X||_F} - 1 \right) \quad \geq \quad 0 , 
\end{equation}
where $X$ is the best rank-$r$ unconstrained approximation of $M$. 
It tell us how far away, in percent, the semi-NMF $UV$ is from the best unconstrained solution $X$. Note that 
$\text{quality}(U,V) = 0$ if and only if $UV$ matches the error of the best rank-$r$ approximation of $M$ if and only if the best rank-$r$ approximation of $M$ is semi-nonnegative.

The Matlab code is available at \url{https://sites.google.com/site/nicolasgillis/}.  All tests are preformed using Matlab on a laptop Intel CORE i5-3210M CPU @2.5GHz 2.5GHz 6Go RAM. We use the function \texttt{linprog} of Matlab to solve the linear systems within the bisection method implemented for Problem~\eqref{yepsi} (we have also implemented a version using CVX \cite{cvx, cvx2}  for users' convenience --note that the solution of~\eqref{yepsi} in $y$ is non-unique and hence the solutions generated by different solvers are usually different).

\subsection{Nonnegative and Semi-Nonnegative Matrices} 

In order to confirm our theoretical findings from Section~\ref{sec3}, 
namely that A3 computes 
optimal solutions for nonnegative matrices (given that $M^TM$ is irreducible; see Corollary~\ref{irreduc}), 
and for many semi-nonnegative matrices (under a certain condition; see Theorem~\ref{th4}), we generate matrices as follows: 
\begin{enumerate}
\item \emph{Nonnegative matrices}. We generate each entry of $M$ with the uniform distribution in the interval [0,1], that is, 
we use \texttt{M =  rand(m,n)}. Note that, with probability one, $M > 0$ hence $MM^T$ is irreducible. 

\item \textit{Semi-nonnegative matrices of rank higher than $r$}. We generate matrices for which $k = \rank(M) = \rank_s(M) = r+10$: we take $M = UV$ where each entry of $U$ is generated with the normal distribution (mean 0, variance 1) and each entry of $V$ with the uniform distribution in the interval $[0,1]$, that is, we use \texttt{M = randn(m,k)*rand(k,n)}.  
\end{enumerate}

For each value of $r$ (20, 80), we generate 500 such matrices and Figure~\ref{numexp1} displays the box plots of the measure defined in Equation~\eqref{normrelerr} (we perform a single initialization for each generated matrix). 
These results confirm that A3 performs perfectly for these types of matrices. 
Note that 
(i) A2 performs relatively poorly and lead to solutions worse than RD/100 and KM/100, and 
(ii) RD/100 (resp.\@ KM/100) do not always generate solutions close to optimality, in particular for semi-nonnegative matrices when $r=80$ for which the average quality is 5.5\% (resp.\@ 1.9\%). 
\begin{figure}[ht!]
\begin{center}
\begin{tabular}{cc}
\includegraphics[width=6cm]{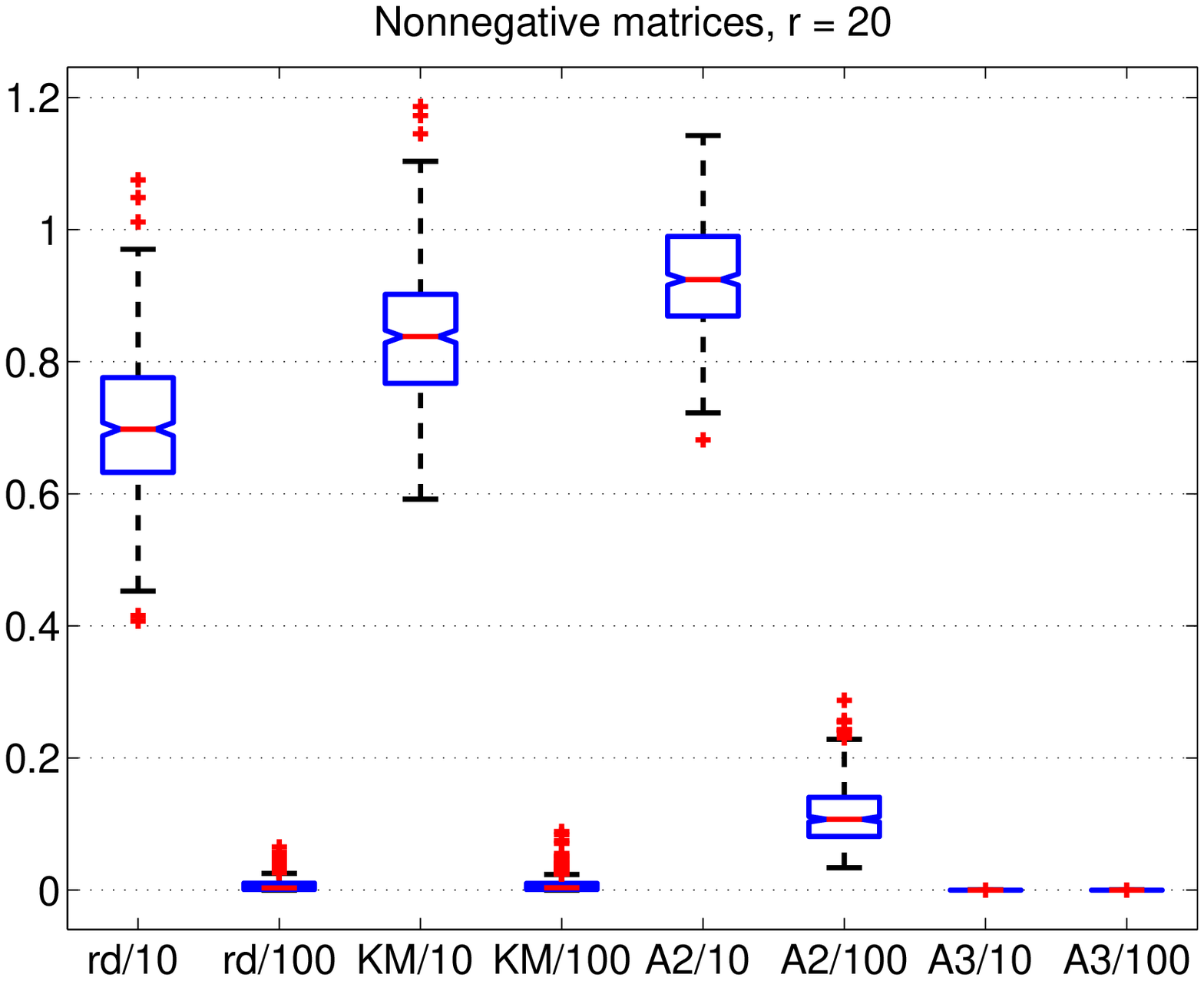} 
 \quad  & \quad
\includegraphics[width=6cm]{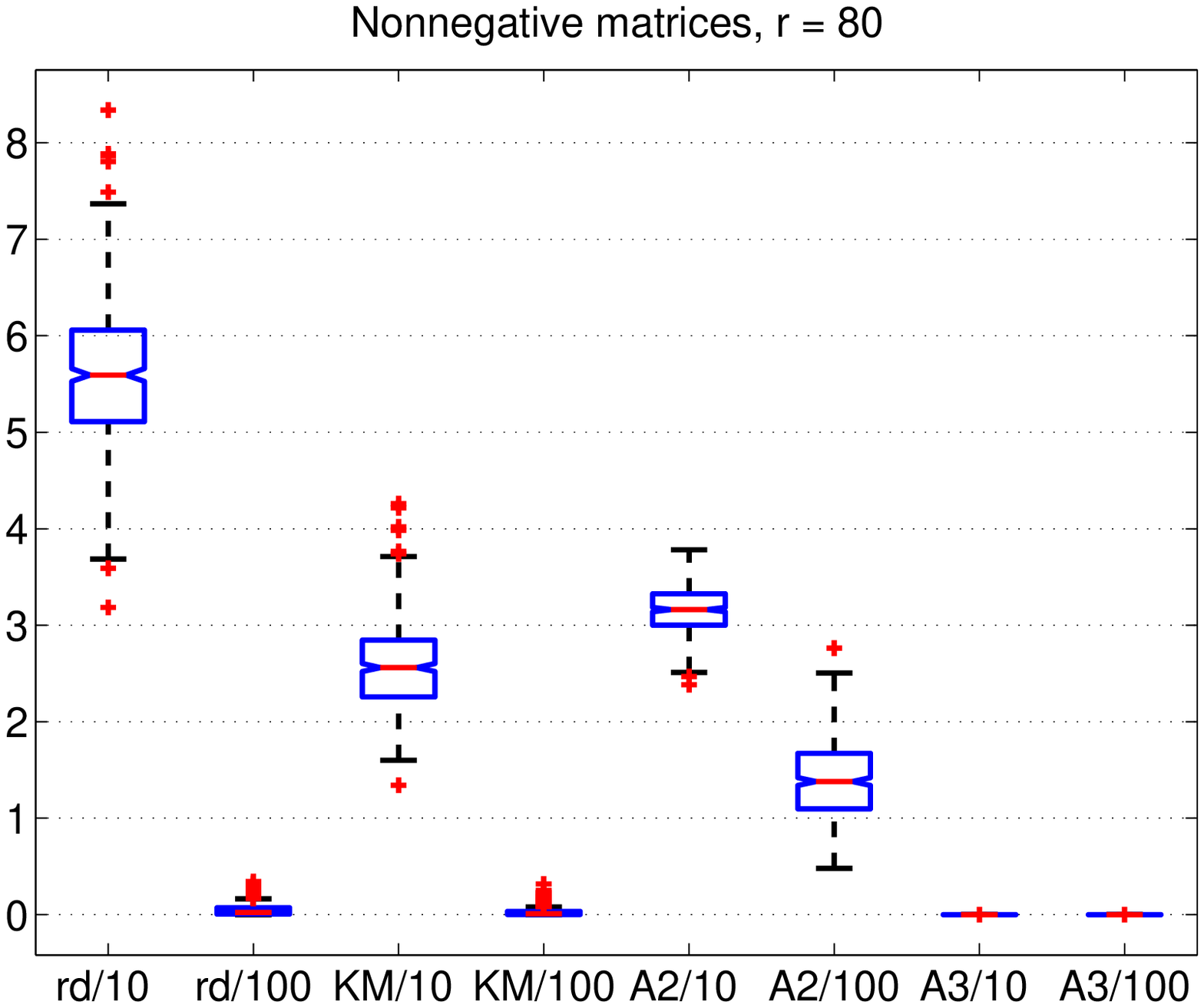} \\
 \includegraphics[width=6cm]{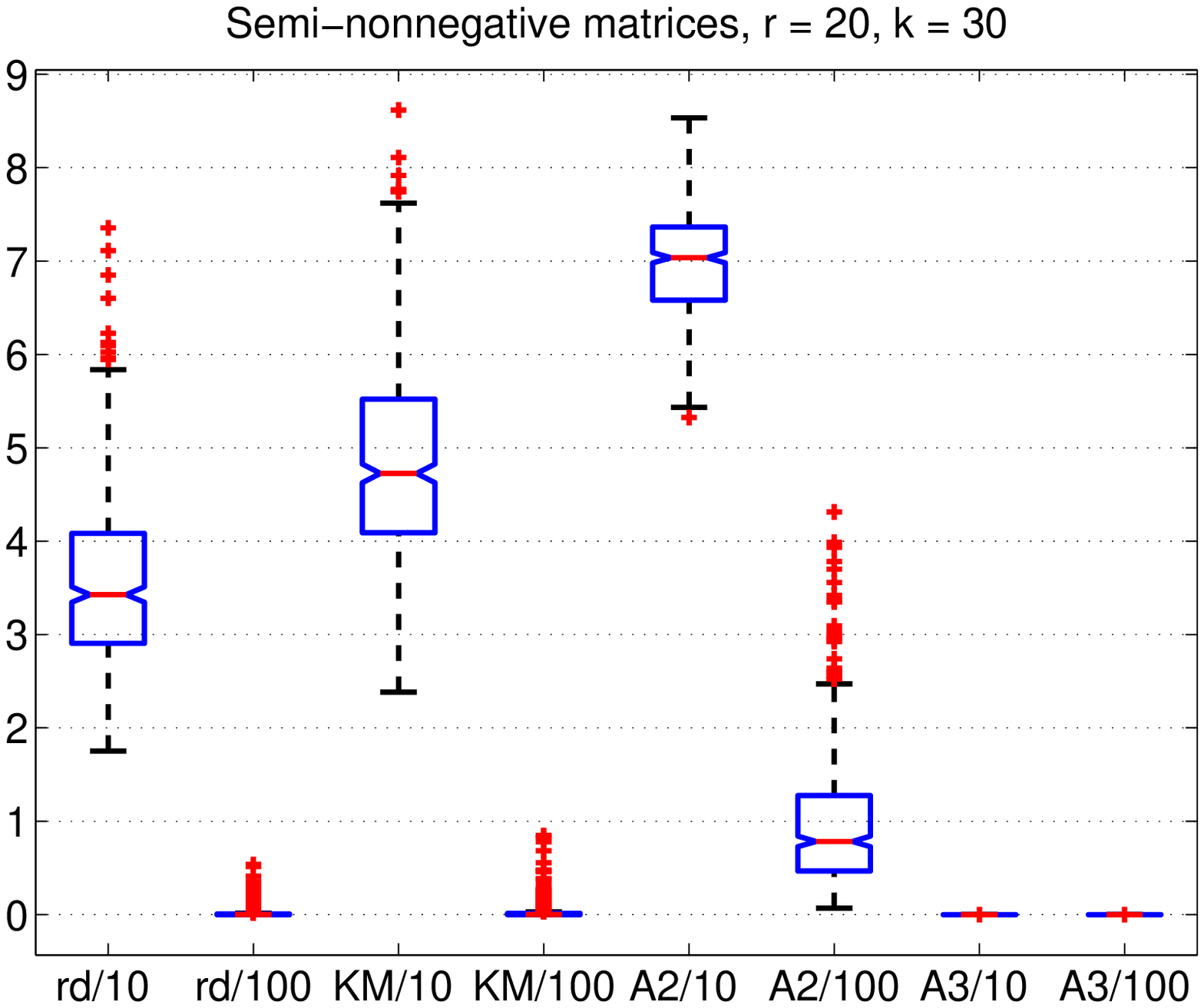}
 \quad  & \quad \includegraphics[width=6cm]{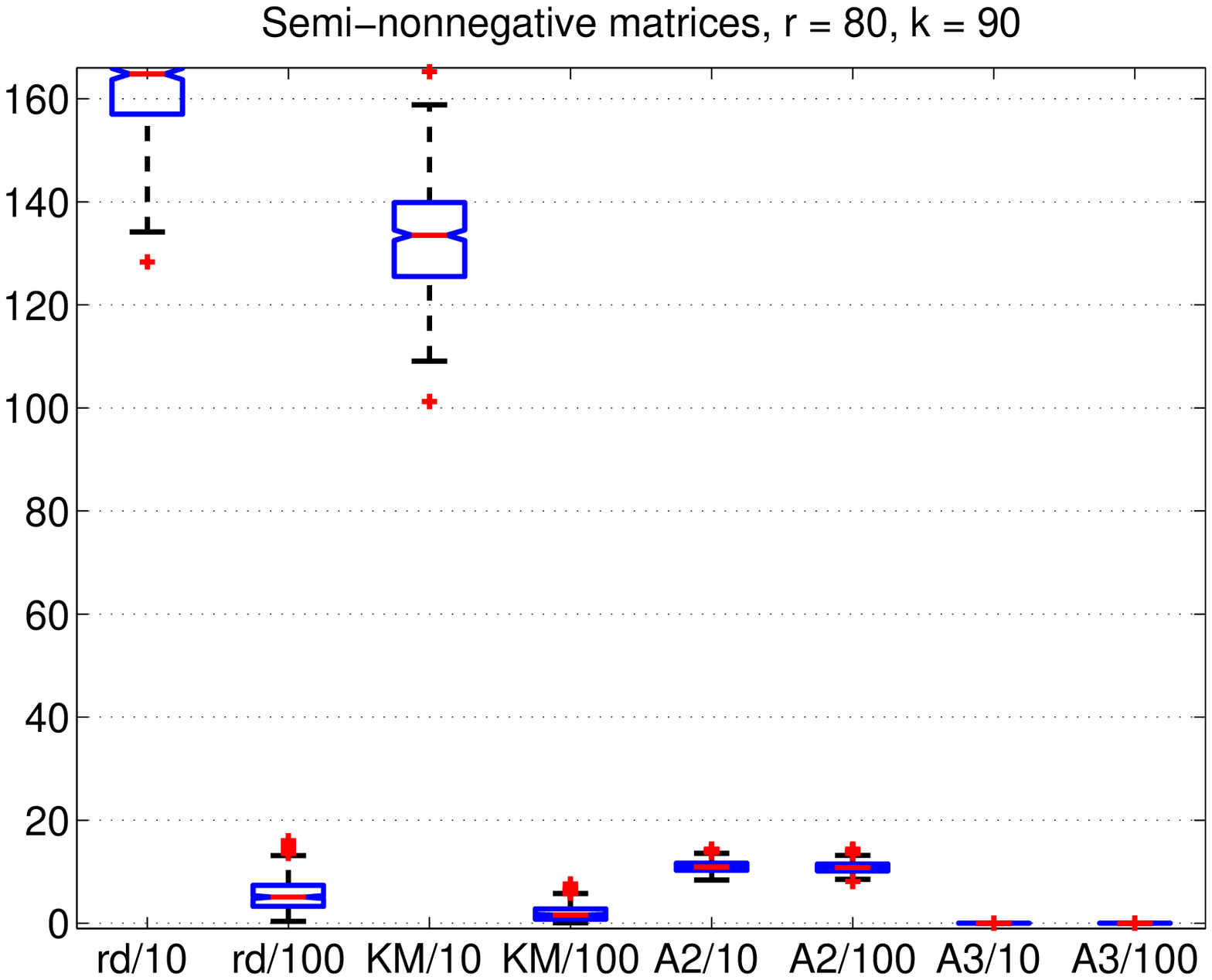}
\end{tabular}
\caption{Box plots of the error defined in Eq.~\eqref{normrelerr} for nonnegative and semi-nonnegative matrices.}
\label{numexp1}
\end{center}
\end{figure}

\subsection{Semi-Nonnegative Matrices plus Noise} 

In this subsection, we generate semi-nonnegative matrices with $\rank(M) = \rank_s(M) = r$ to which we add Gaussian noise. 
First we compute $M = UV$ where each entry of $U$ is generated with the normal distribution (mean 0, variance 1) and each entry of $V$ with the uniform distribution in the interval $[0,1]$, that is, we use \texttt{M =  randn(m,r)*rand(r,n)}, 
similarly as in the previous subsection.  
Then we compute the average of the absolute values of the entries of $M$: $x_M = \frac{1}{mn} \sum_{i,j} |M(i,j)|$ and add Gaussian noise proportional to $x_M$: 
we generate $N = \delta \, x_M \, \texttt{randn(m,n)}$ where $\delta$ is the noise level 
and then update $M \leftarrow M + N$. 
For $\delta = +\infty$, each entry of $M$ is generated using the normal distribution, that is, $\texttt{M = randn(m,n)}$. 
Figure~\ref{numexp2} display the box plots of the measure defined in Eq.~\eqref{normrelerr} for different values of $\delta$. As in the previous subsection, for each experiment, we generate 500 matrices and report the quality obtained for a single initialization for each algorithm. 
\begin{figure}[ht!]
\begin{center}
\begin{tabular}{cc}
\includegraphics[width=6cm]{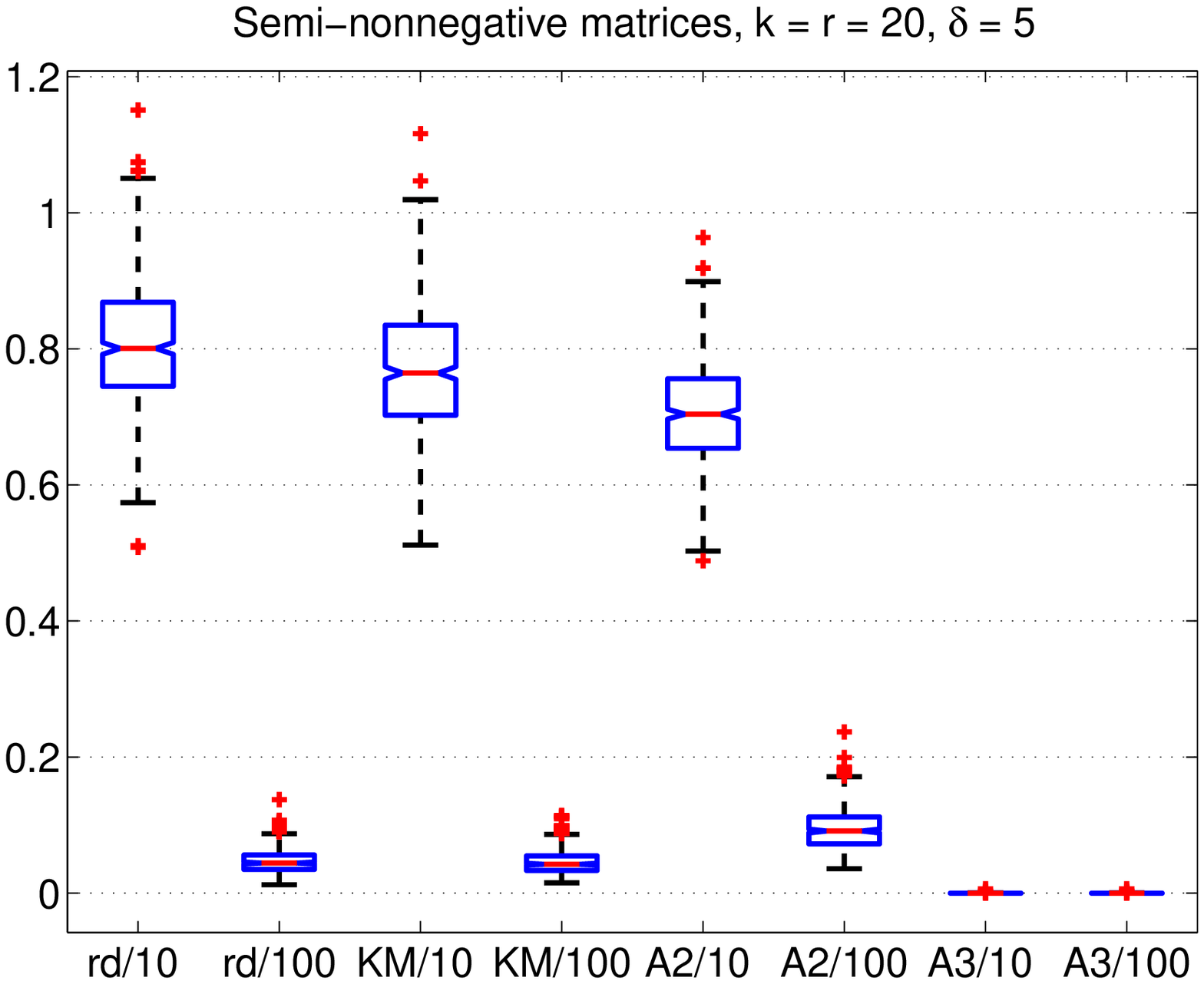} 
 \quad  & \quad
\includegraphics[width=6cm]{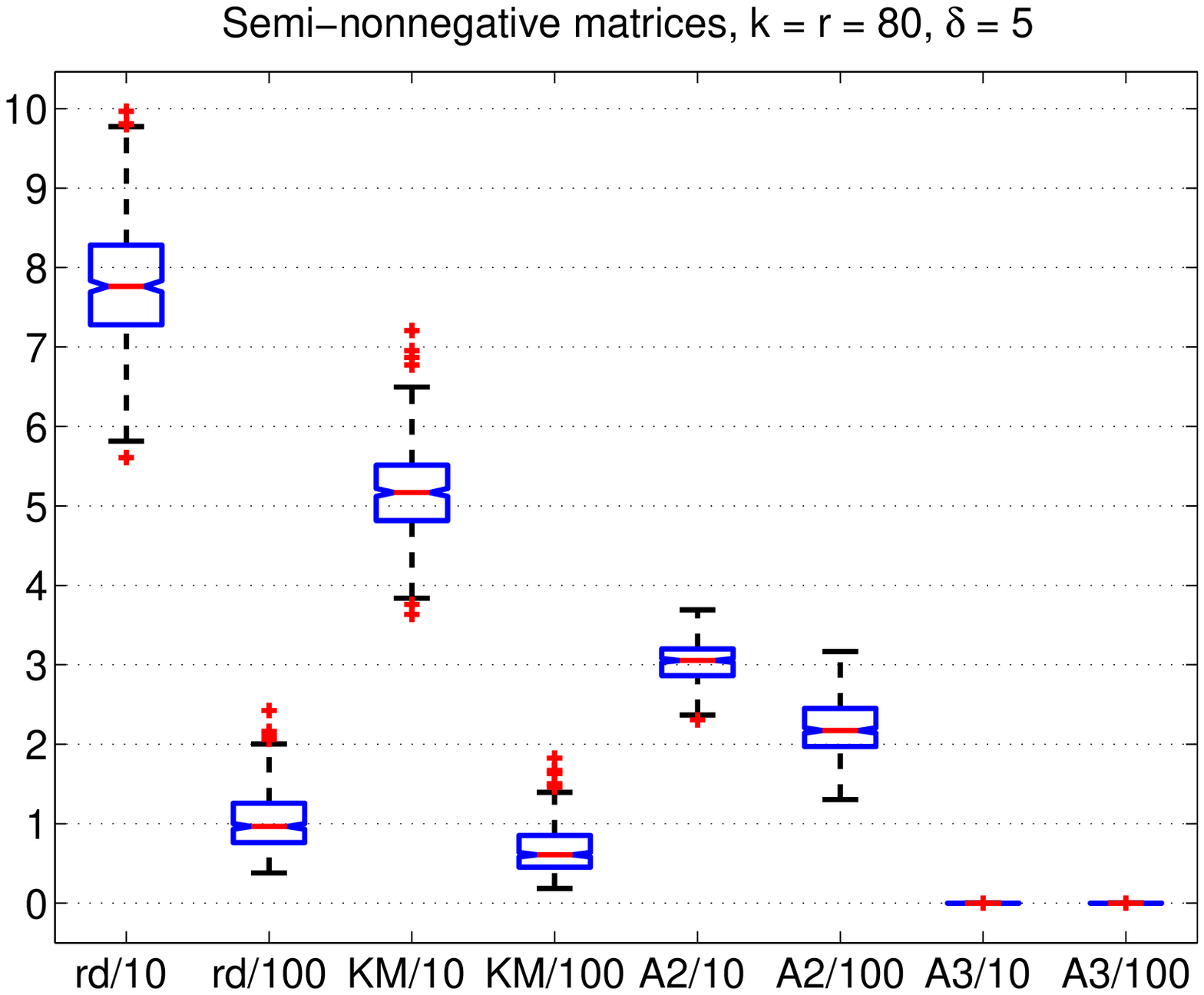} \\
\includegraphics[width=6cm]{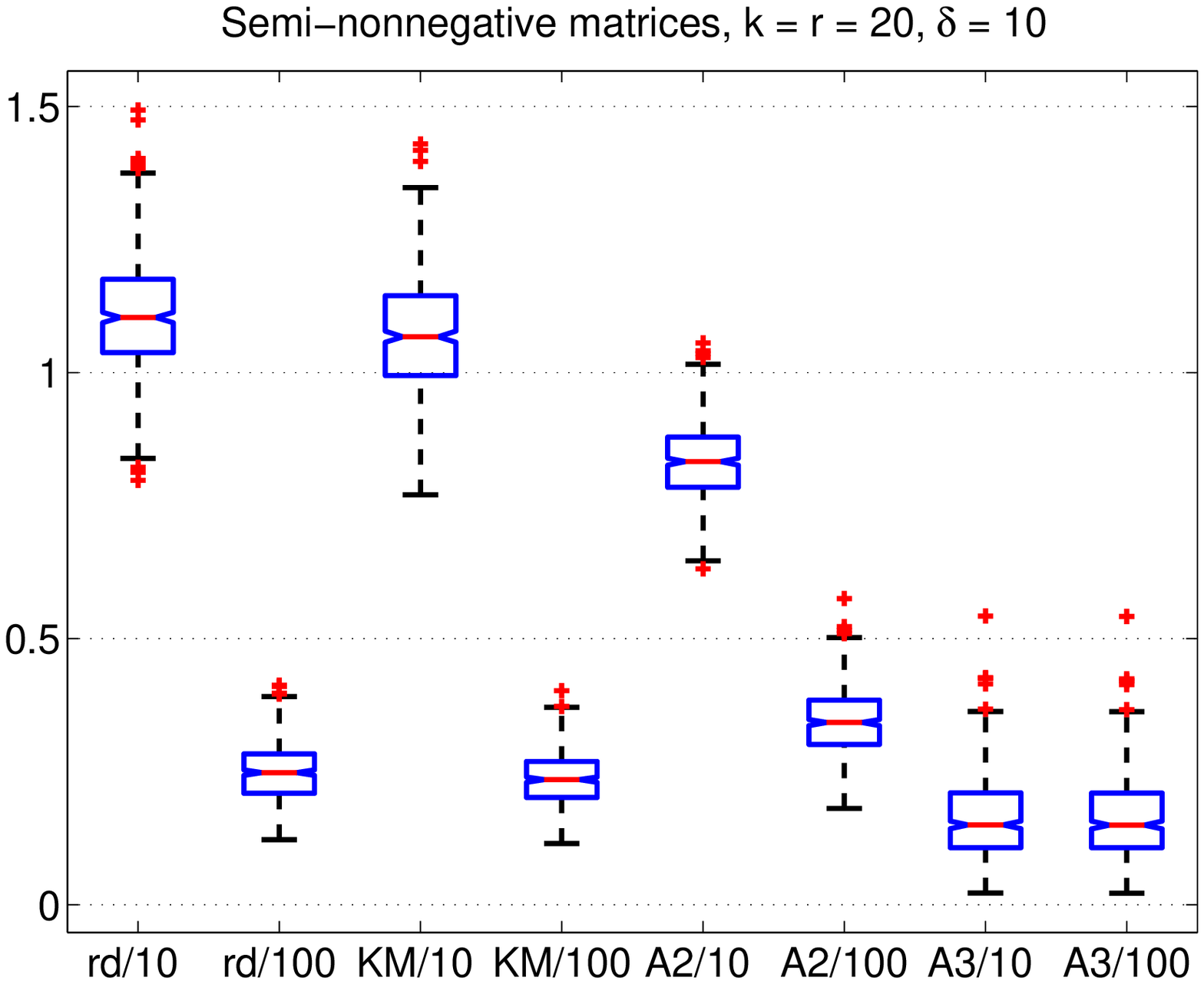} 
 \quad  & \quad
\includegraphics[width=6cm]{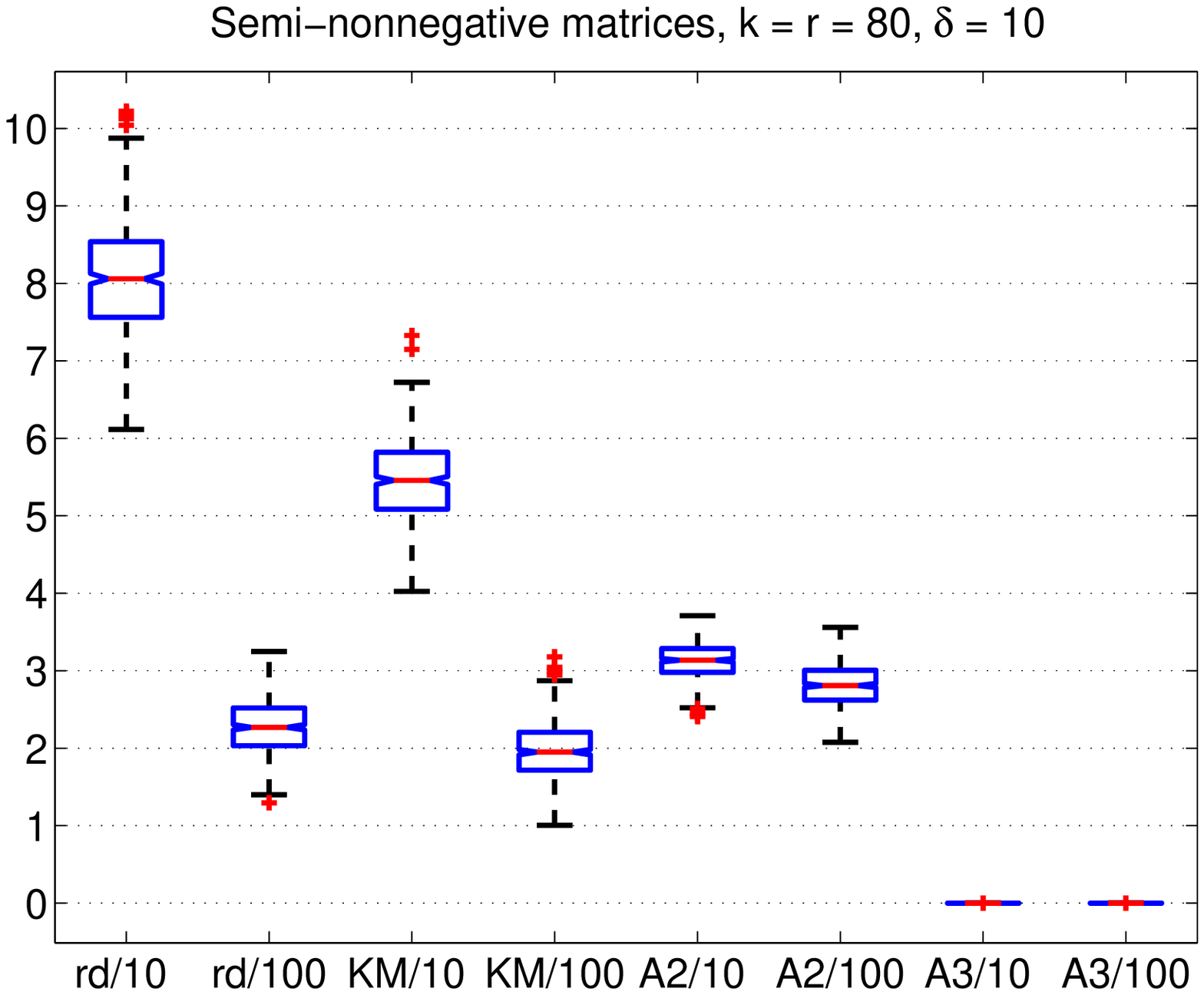} \\
 \includegraphics[width=6cm]{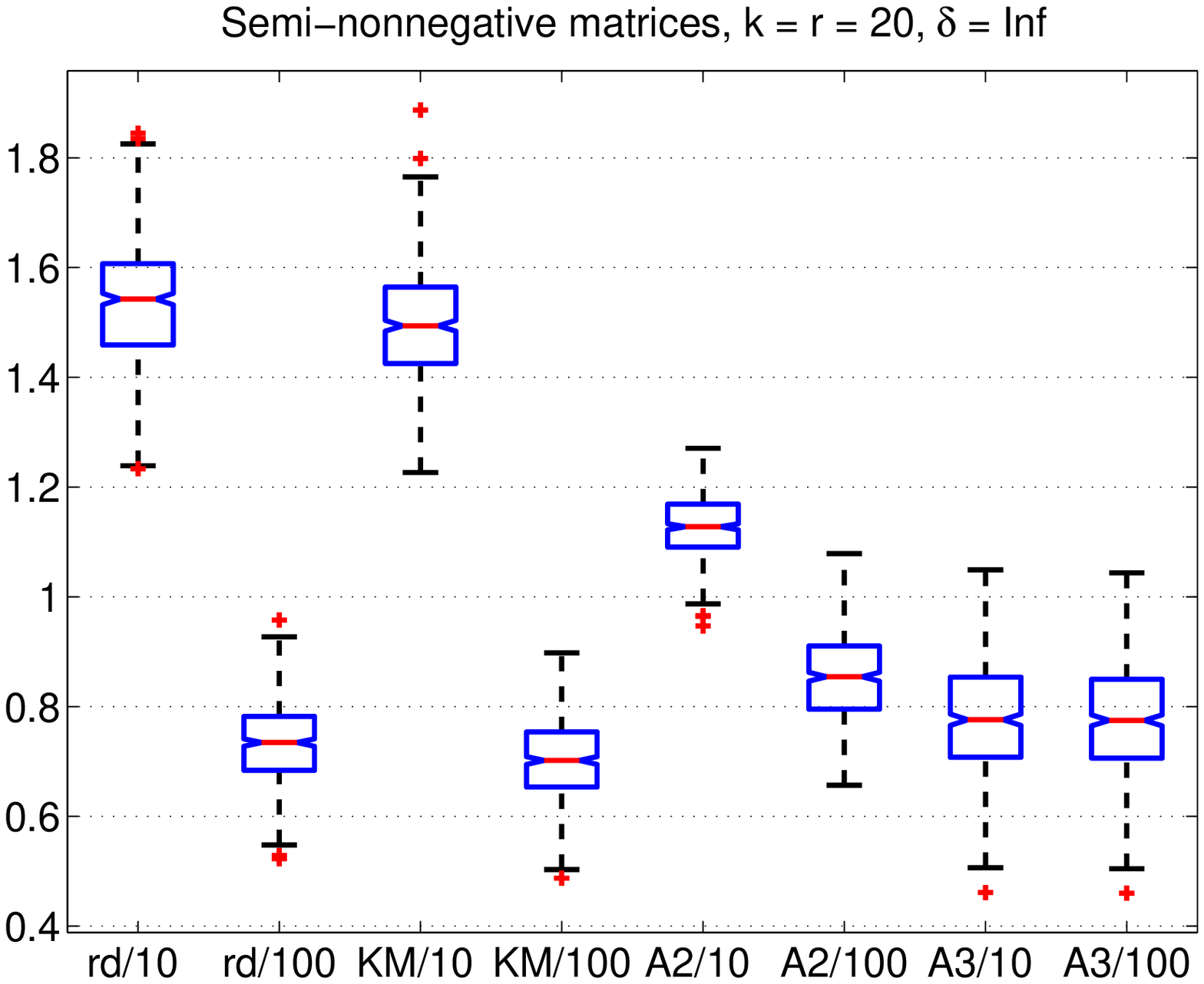}
 \quad  & \quad \includegraphics[width=6cm]{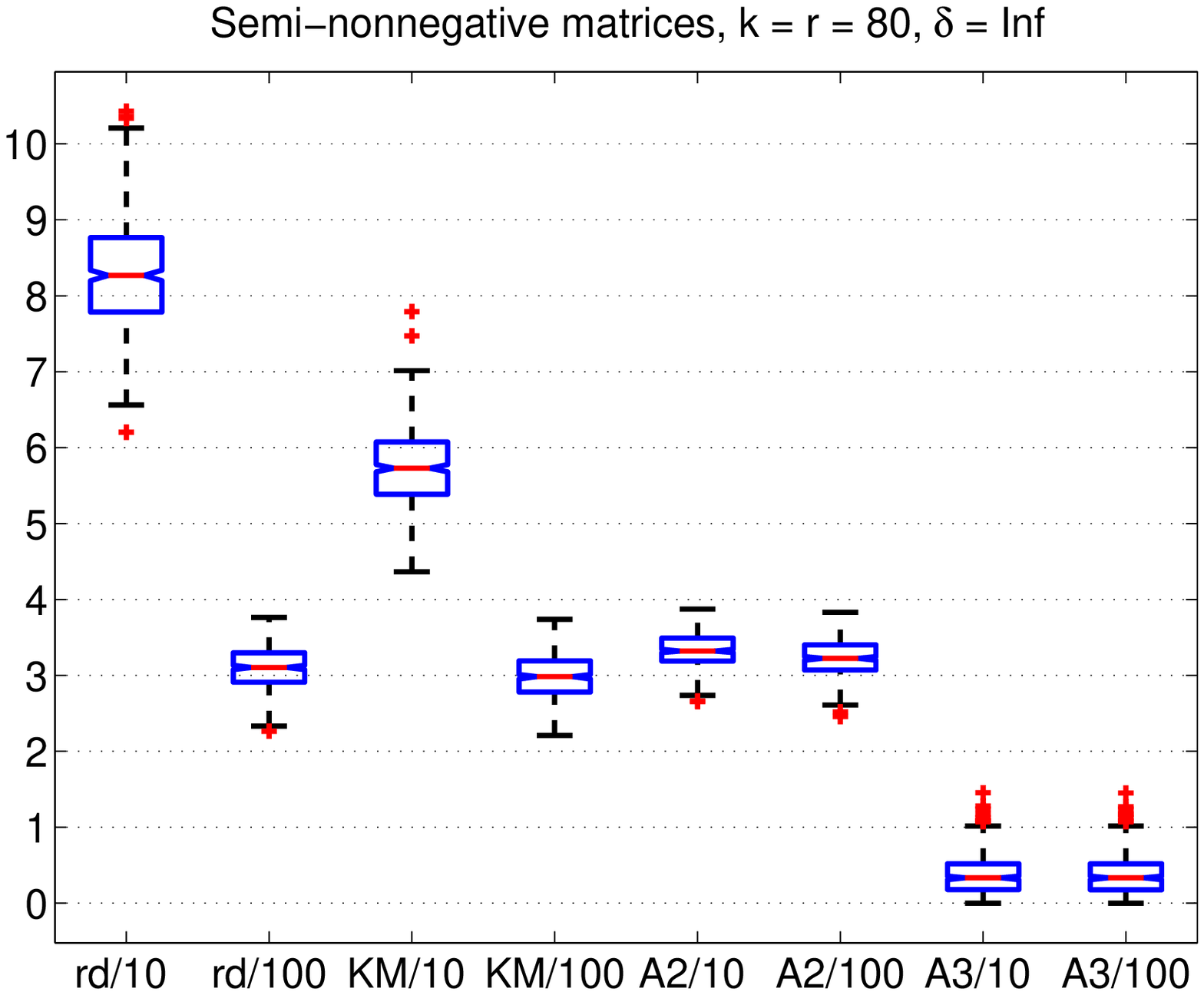} 
\end{tabular}
\caption{Box plots of the error from Eq.~\eqref{normrelerr}: 
semi-nonnegative matrices with noise ($\delta = 5, 10, +\infty$).}
\label{numexp2}
\end{center}
\end{figure} 
We observe that: 
\begin{itemize} 
 
\item A2 performs quite poorly: it is always dominated by RD/100 and KM/100. 
Hence, although A2 is appealing from a theoretical point of view, it does not seem to have much practical use.

\item KM performs in average slightly better than RD, hence k-means initialization seems beneficial in some cases (this will not be the case for some real data sets tested in Section~\ref{realdata}).

\item Even for a relatively large noise level (in particular, $\delta = 5$), A3 performs perfectly (all semi-NMF have quality~\eqref{normrelerr} smaller than $10^{-2}$) because the best rank-$r$ approximation of $M$ is close to being semi-nonnegative.  
The reason why A3 works perfectly even for very large noise levels can be explained with the way the matrix $V$ was generated: using the uniform distribution for each entry. This makes the columns of matrix $UV$ be far from the boundary of a well-chosen half space, that is, there exists $z$ such that $(UV)^Tz/||z||_2 \gg 0$. 
In particular, taking $z = (U^{\dagger})^T e$, the expected value of $(UV)^Tz = V^T e$ is equal to $\frac{r}{2} e$.  

\item When the noise level increases and the rank is not sufficiently large (the first example is for $\delta = 10$ and $r = 20$), the condition of Theorem~\ref{th4} is not (always) met and some solutions generated by A3 do not match the error of the best rank-$r$ approximation. 

\item When only noise is present and the matrix is Gaussian ($\delta = +\infty$), 
the condition of Theorem~\ref{th4} is never met and A3 fails most of the time to extract a semi-NMF whose error is close to the error of the best rank-$r$ approximation. This is not surprising as these matrices are not likely to have semi-nonnegative best rank-$r$ approximations.   
 
However, when $r$ is large ($r = 80$), A3 outperforms RD, KM and A2; see Figure~\ref{numexp2} (bottom right) and Figure~\ref{numexp3} (right). 
We believe the reason is that, although the best rank-$r$ approximations are not semi-nonnegative, they have a vector in their row space close enough to the nonnegative orthant hence A3 provides a good initial approximation. 

On the contrary, when $r$ small ($r = 20$), A3 performs worse than RD/100 and KM/100 although the gap between both approaches is not significant as shown in Figure~\ref{numexp2} (bottom left); see also Figure~\ref{numexp3}. 
(Note that, with $r = 10$, A3 performs even worse compared to RD/100 and KM/100; see also Section~\ref{realdata} for an example of this situation on real data.) 

Note that we have observed this behavior for other values of $m,n$ and $r$. 
As explained above, the reason is that it becomes more likely as $r$ increases for the row space of the best rank-$r$ approximation of $M$ to contain a vector close to the nonnegative orthant hence for A3 to perform well. (At least, the row space cannot get further away from the nonnegative orthant as it is expanded as $r$ increases.)   This will be conformed on real data in Section~\ref{realdata}.

\item It seems that A3 generates matrices close to stationary points of \eqref{semiNMF} as the difference between the boxplots of A3/10 and A3/100 is small on all these examples. This is another advantage of A3: 
in all the experiments we have performed, it always allowed Algorithm~\ref{cdsemi} to converge extremely quickly (essentially within 10 iterations); see Figure~\ref{numexp3} displaying the average value of the quality~\eqref{normrelerr} for each iteration on the 500 Gaussian matrices. This observation will be confirmed on real data in Section~\ref{realdata}. 
\begin{figure}[ht!]
\begin{center}
\begin{tabular}{cc}
\includegraphics[width=6.1cm]{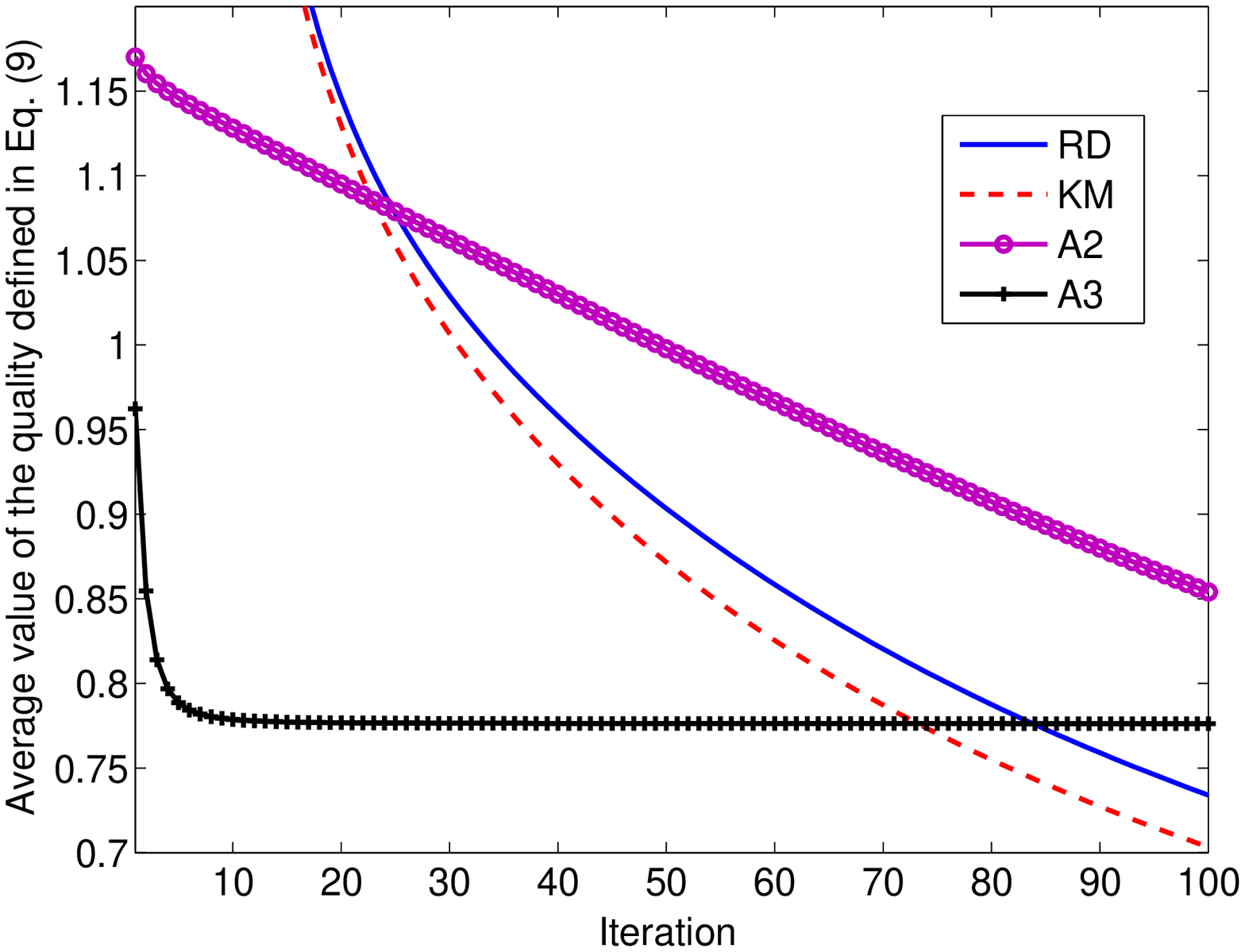} 
   & 
\includegraphics[width=6.2cm]{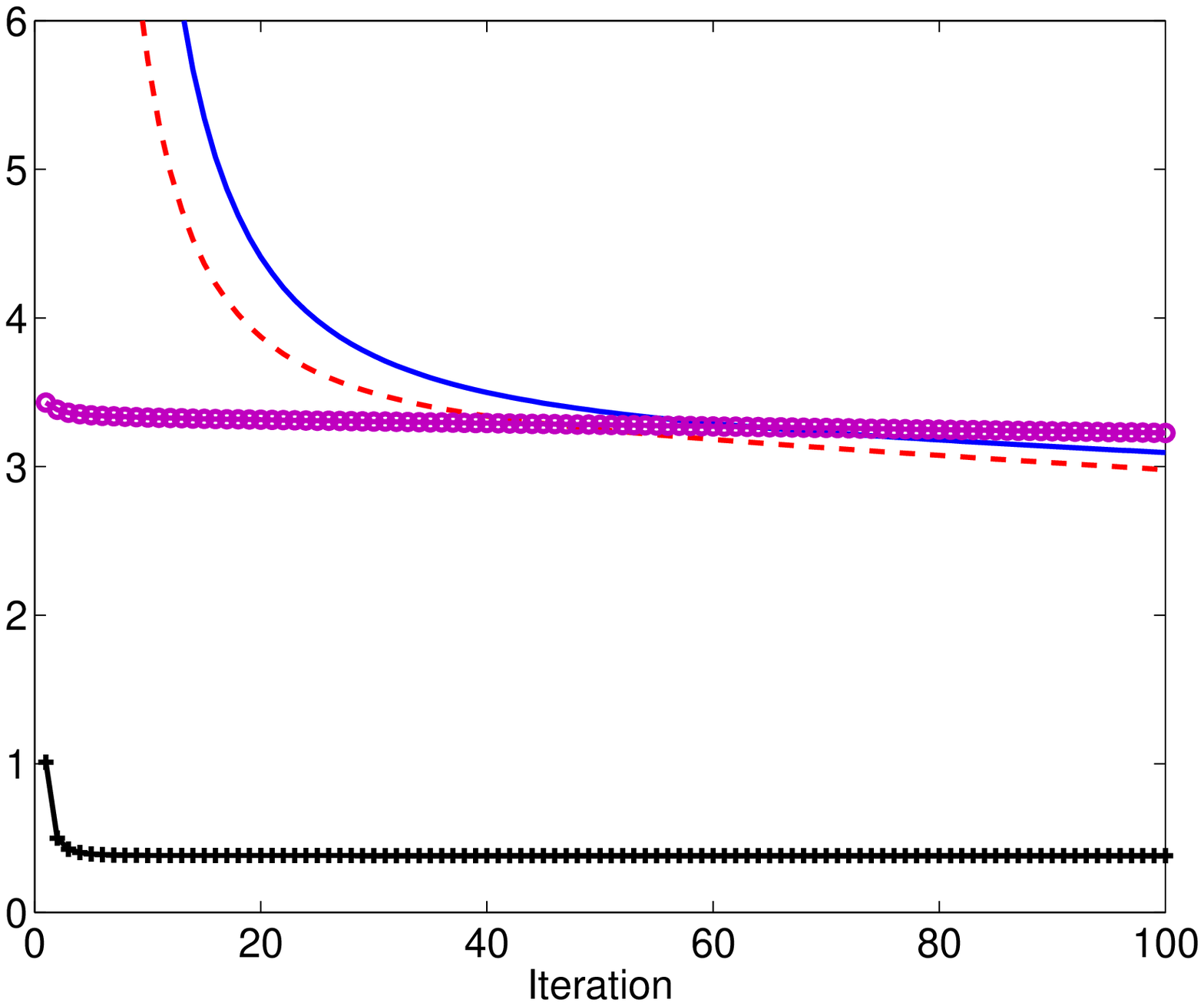} 
\end{tabular}
\caption{Average value of the quality~\eqref{normrelerr} for each iteration of Algorithm~\ref{cdsemi} on the 500 Gaussian matrices ($\delta = +\infty$) for $r = 20$ (left) and $r = 80$ (right) with the different initialization strategies.}
\label{numexp3}
\end{center}
\end{figure} 

\end{itemize}

Finally, the two main recommendations that we can give based on these experiments are the following: 
\begin{enumerate} 

\item A3 works in general very well, being optimal for matrices whose best rank-$r$ approximation is semi-nonnegative. 
Hence we would always recommend to try it on your favorite matrices. 

\item RD and KM sometimes work better than A3 (after sufficiently many iterations of a semi-NMF algorithm), 
in particular when the factorization rank $r$ is small and the best rank-$r$ approximation of the input matrix is far from being semi-nonnegative (in fact, if the best rank-$r$ approximation would be semi-nonnegative, A3 would perform perfectly). 
They should be tried on matrices for which A3 was not able to compute a semi-NMF whose error is close to the error of the best rank-$r$ approximation.   

\end{enumerate} 

\begin{remark}
In these experiments, we intentionally took $n \geq m$ because an $m$-by-$n$ matrix is less likely to be semi-nonnegative when $n \geq m$. 
In particular, an $m$-by-$n$ Gaussian matrix with $n \leq m$ is semi-nonnegative with probability one since $\rank(M) = \rank_s(M) = n$ with probability one (although this does not imply that its best rank-$r$ approximation is; see Theorem~\ref{th4}). 
For example, running exactly the same experiment with $m = 200, n = 100, r=20$ and $\delta = 10$, 
86\% of the solutions generated by A3 match the best rank-$r$ approximation up to $0.01\%$ 
(while only 0.4\% do for $m = 100, n = 200, r=20$; see Figure~\ref{numexp2}). 
\end{remark}

\subsection{Real Data} \label{realdata}

In this section, we compare the different approaches on three data sets: 
\begin{itemize}
\item CBCL face data set\footnote{\url{http://cbcl.mit.edu/cbcl/software-datasets/FaceData2.html}}: it is arguably the most popular data set for NMF as it was used in the foundational paper of Lee and Seung~\cite{LS99} with $r = 49$.  
It consists in 2429 facial images, 19-by-19 pixels each. 
The corresponding matrix $M$ therefore has size 361 by 2429 and is nonnegative. This will illustrate the optimality of A3 on nonnegative data. 

\item  Inonosphere and Waveform UCI data sets: 
these are the two data sets that contain both positive and negative entries used in~\cite{DTJ10}. 
Inonosphere corresponds to a 34-by-351 matrix with values in the interval [-1,1], Waveform to a 22-by-5000 matrix with values in the interval [-4.2,9.06] and we use $r= 3, 5, 10$ for both data sets; 
see \url{https://archive.ics.uci.edu/ml/datasets.html} for all the details. 
\end{itemize}

Our goal is to illustrate, on real data, the observations made on synthetic data sets. Again we compare the four semi-NMF initializations (RD, KM, A2 and A3) combined with Algorithm~\ref{cdsemi} in terms of the quality measure defined in Equation~\eqref{normrelerr}. 
For RD and KM, we use ten initializations (k-means generates in most cases different solutions for different runs) 
and report both the average quality and the best quality obtained by the different runs.  
Table~\ref{numreal} reports the error after 10 and 100 iterations of the four approaches (as before). 
\begin{center}
\begin{table}[h!]
\begin{center}
\caption{Numerical results on real-world data sets: quality~\eqref{normrelerr} of the different initialization approaches for semi-NMF. }
\label{numreal}  
\begin{tabular}{|c|c|c|c|c|c|c|c|}
\hline 
      & CBCL &  Ion 3 &  Ion 5 & Ion 10 & Wave 3 & Wave 5 & Wave 10  \\  \hline
 RD/10 - average & 17.59 & 0.59 & 1.90 & 3.12 & 0.19 & 2.43 & 3.44 \\
 RD/10 - best & 16.73 & 0.47 &  1.49 & 2.63 & 0.15 &  1.96 & 2.90 \\
 KM/10 - average & 27.44 & 0.54 & 2.11 & 4.22 & 0.23  & 5.40 & 14.87  \\
 KM/10 - best  & 26.13 & \textbf{0.41} & 1.99 & 2.68 & 0.23 & 5.24 &  13.88 \\
 A2/10  & 1.18 & 0.63 & 3.04 & 3.65 & 0.56 & 2.09 & 3.50 \\
 A3/10 &  \textbf{0} & 0.67 & \textbf{0.38} & \textbf{0} &  \textbf{0} & \textbf{0}  & \textbf{0}  \\  \hline  
 RD/100 - average &  1.59 & 0.16  & 0.44 & 0.37 & 0.01 & 0.03  & 0.07 \\
 RD/100 - best & 1.34 & \textbf{0.15} & \textbf{0.29} & 0.33 & 0.01 & 0.01 & 0.03  \\ 
 KM/100 - average & 6.00   & 0.16 & 0.98 & 0.44 & 0.01 & 0.09 &  0.13 \\
 KM/100 - best & 5.61 & \textbf{0.15} & 0.31 & 0.39 & 0.01 & 0.05& 0.06  \\  
 A2/100 & 1.18 & 0.16 & 0.99 & 1.57 & 0.02 & 0.14 &  0.25 \\  
 A3/100  &  \textbf{0} & 0.67 & 0.38 & \textbf{0} &  \textbf{0} & \textbf{0} &  \textbf{0} \\
\hline
\end{tabular} 
\end{center}
\end{table}
\end{center}
 
Figure~\ref{numion} shows the evolution of the quality for the CBCL data set, 
Figure~\ref{numion} for the Ionosphere data set,  and 
Figure~\ref{numwave} for the Waveform data set. 
\begin{figure}[ht!]
\begin{center}
\begin{tabular}{c}
\includegraphics[width=7cm]{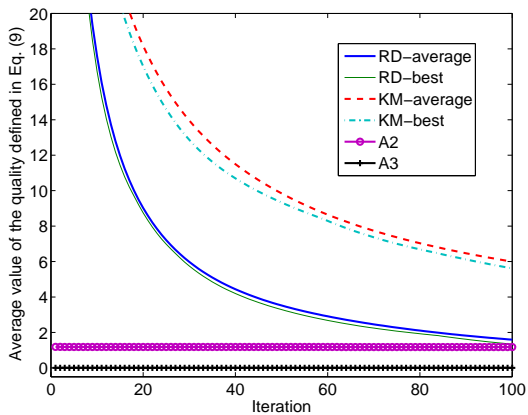} 
\end{tabular}
\caption{Quality~\eqref{normrelerr} for each iteration of Algorithm~\ref{cdsemi} with the different initialization strategies on the CBCL data set ($r = 49$).} 
\label{numcbcl}
\end{center}
\end{figure} 
\begin{figure}[ht!]
\begin{center}
\begin{tabular}{c}
\includegraphics[width=6.1cm]{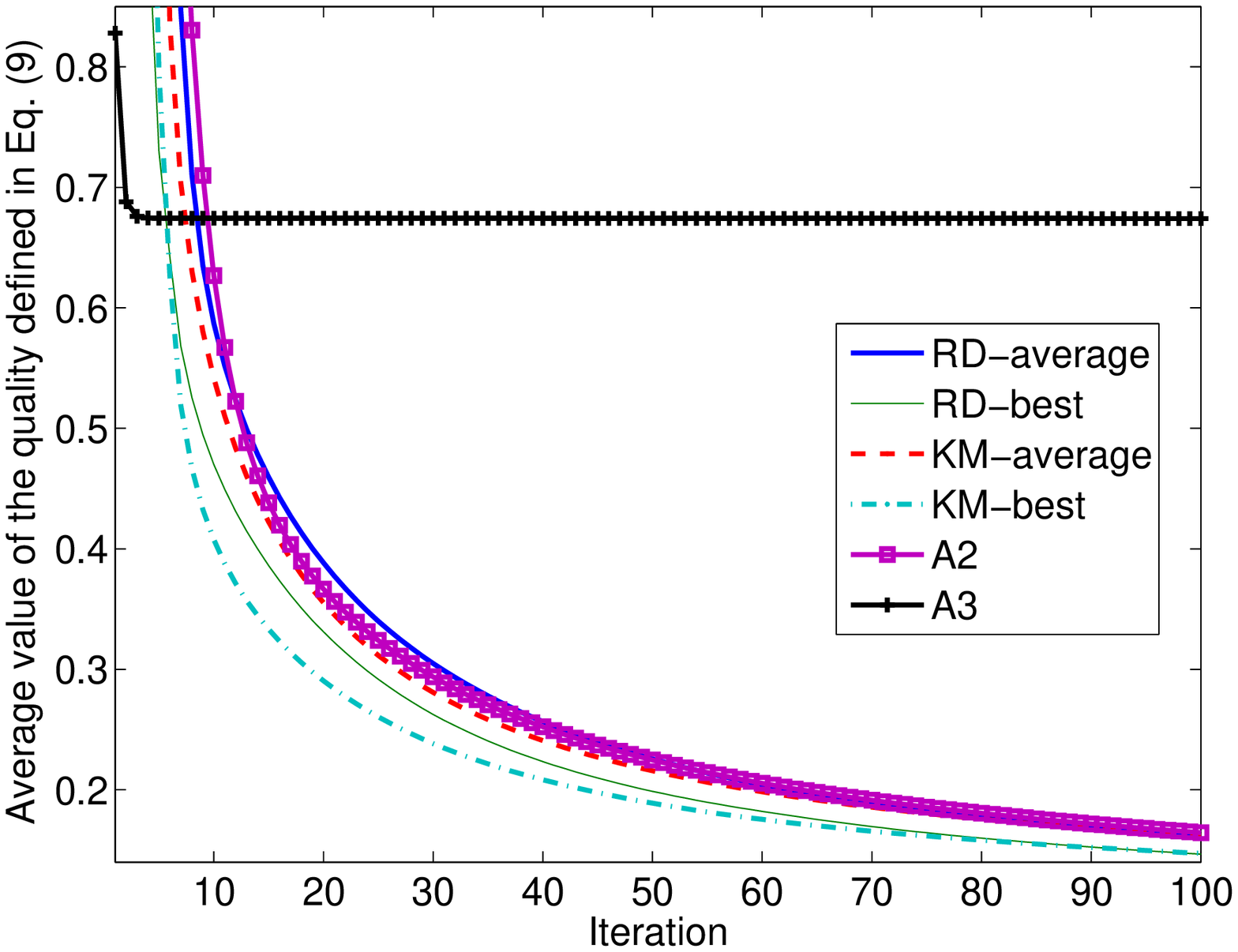} \\ 
\includegraphics[width=6.1cm]{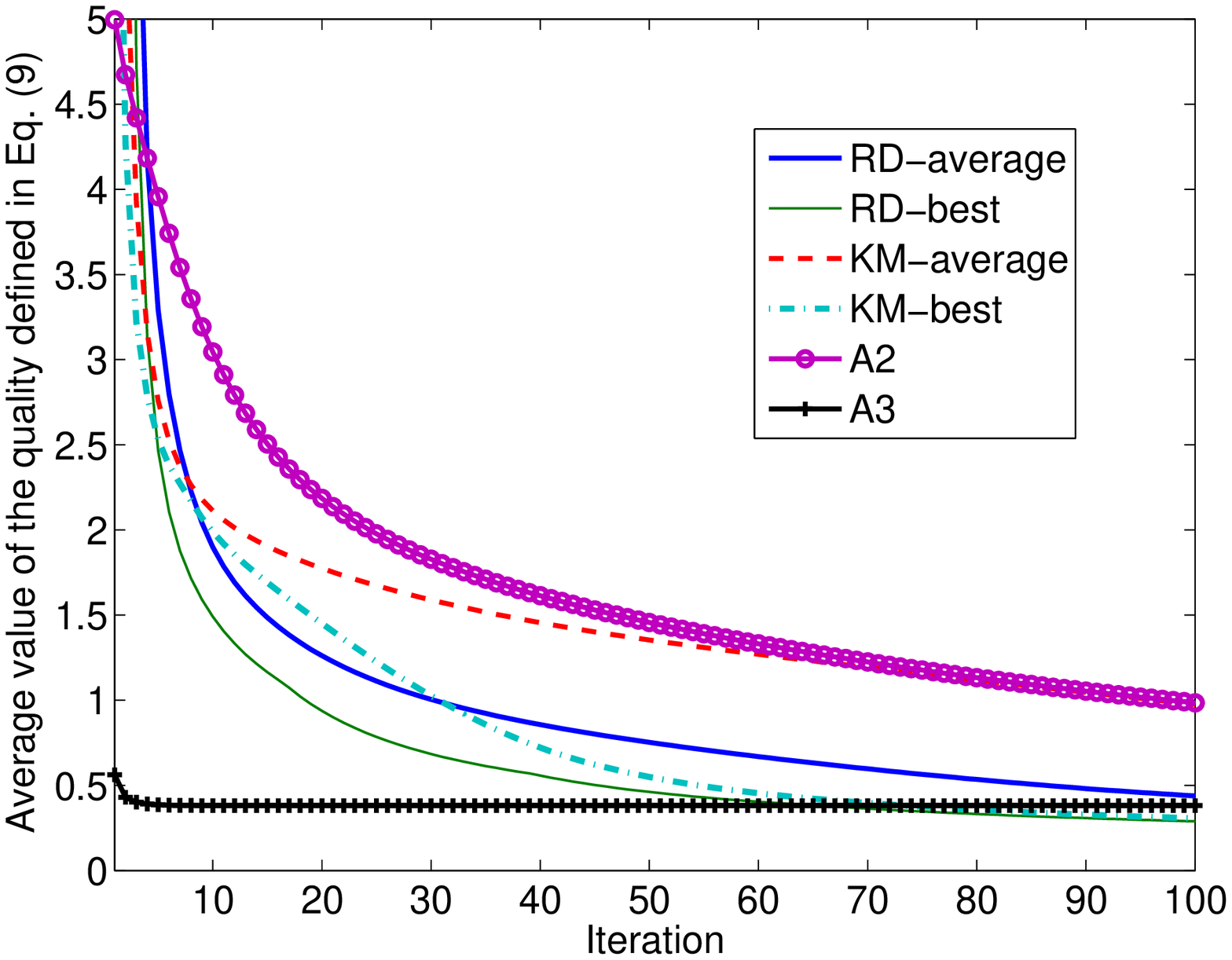} 
\includegraphics[width=6.3cm]{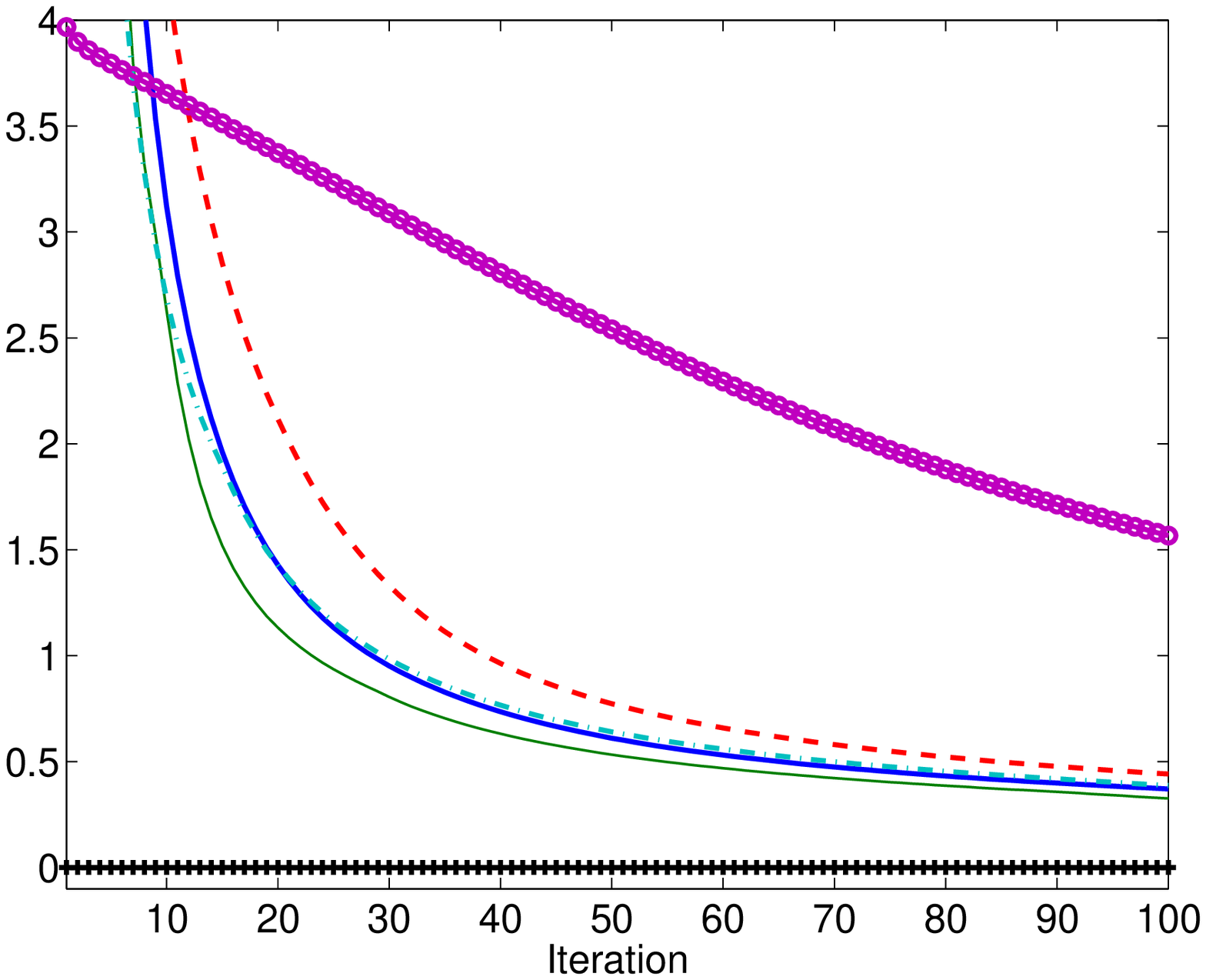} 
\end{tabular}
\caption{Quality~\eqref{normrelerr} for each iteration of Algorithm~\ref{cdsemi} with the different initialization strategies for the Ionosphere data set: $r = 3$ (top), $r = 5$ (bottom left) and $r = 10$ (bottom right).}
\label{numion}
\end{center}
\end{figure} 
\begin{figure}[ht!]
\begin{center}
\begin{tabular}{c}
\includegraphics[width=6.3cm]{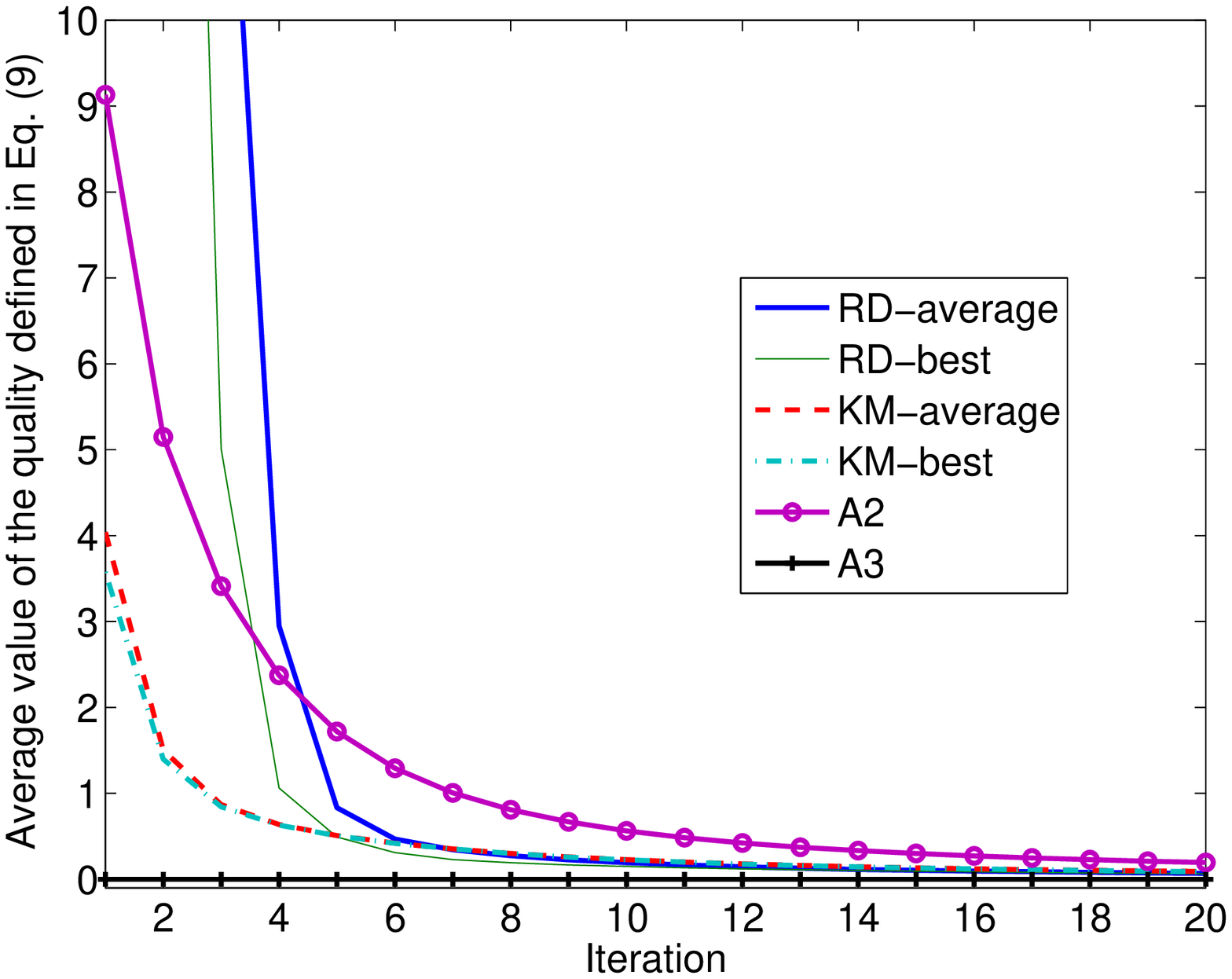} \\ 
\includegraphics[width=6.3cm]{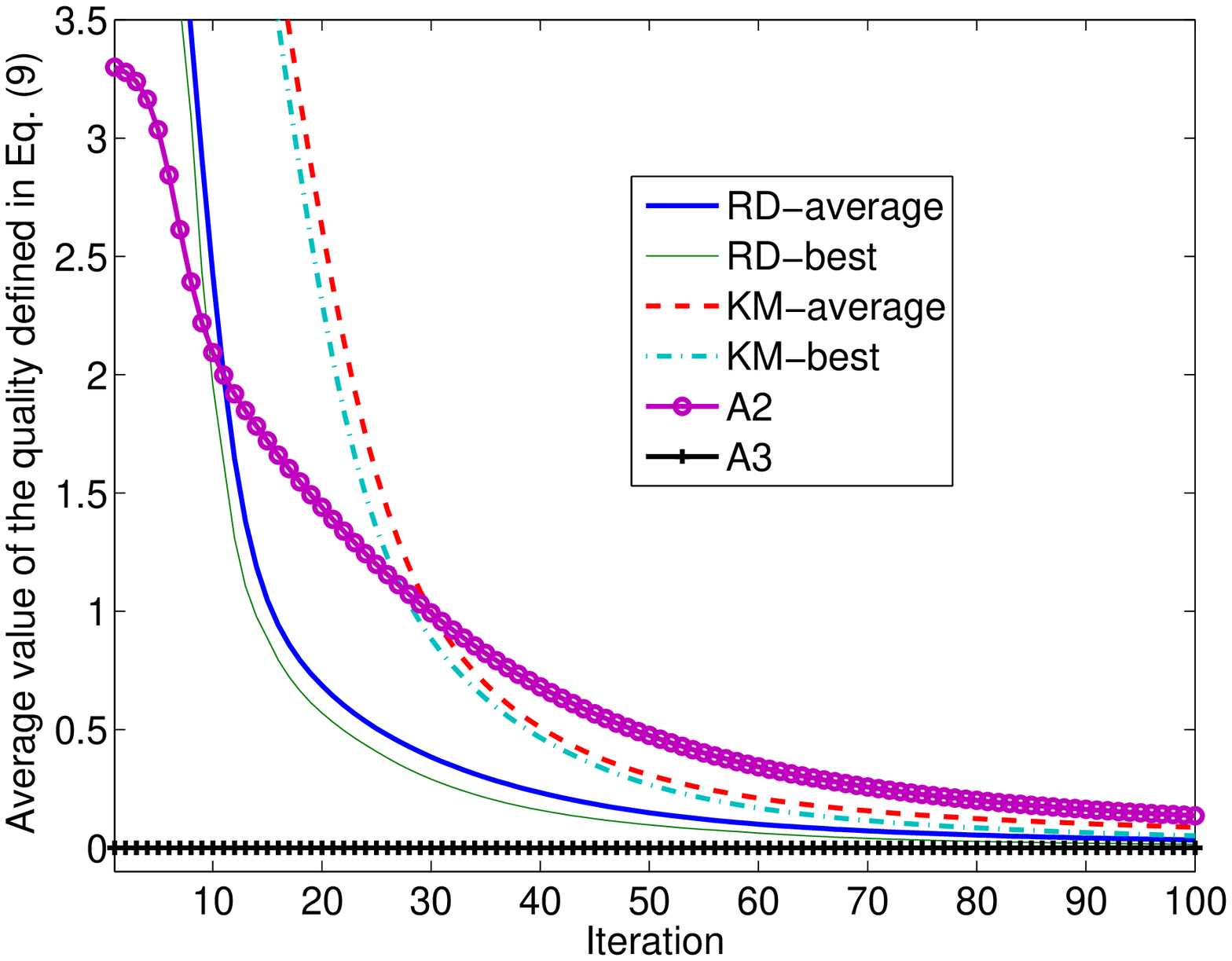} 
\includegraphics[width=6cm]{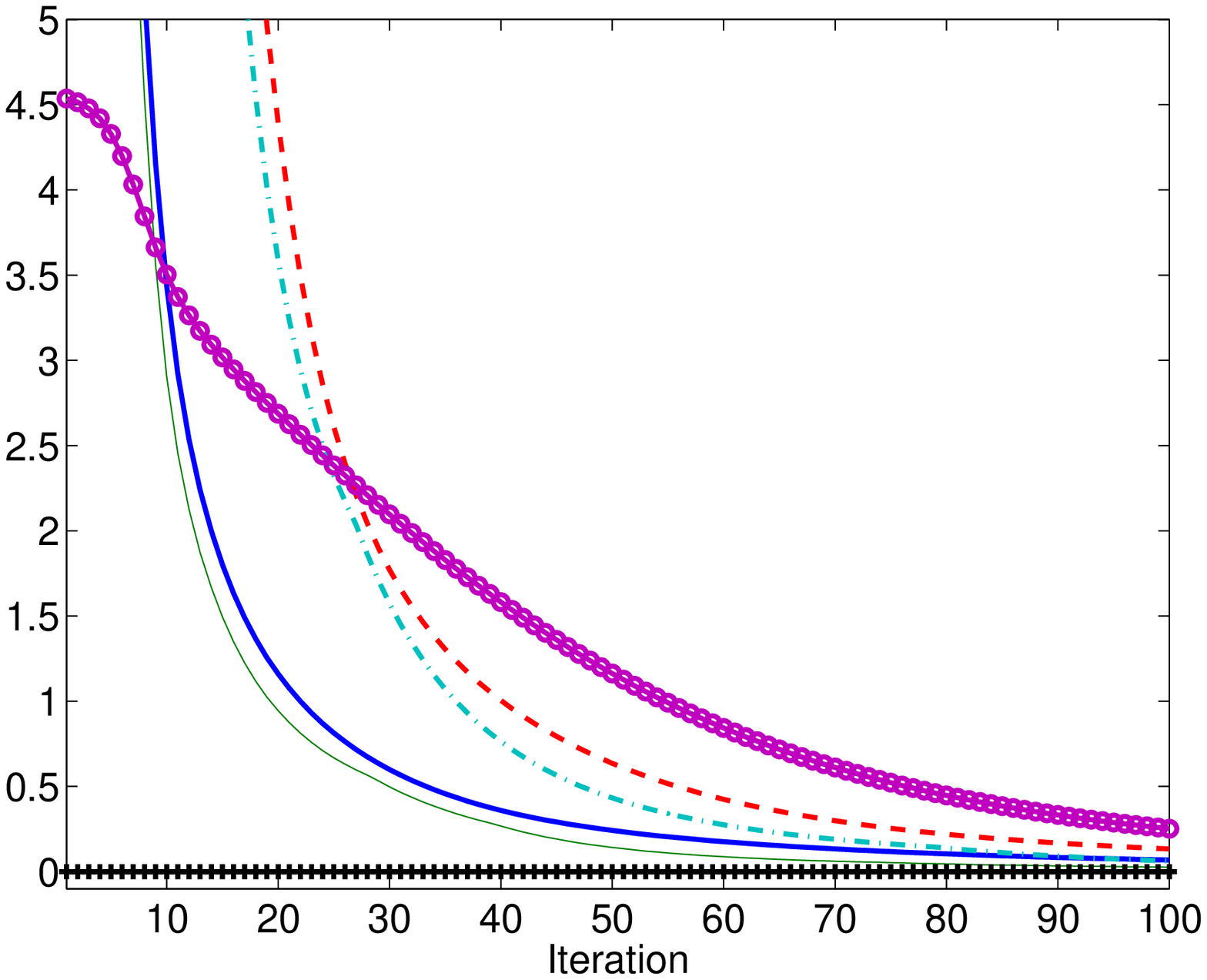} 
\end{tabular}
\caption{Quality~\eqref{normrelerr} for each iteration of Algorithm~\ref{cdsemi} with the different initialization strategies for the Waveform data set: $r = 3$ (top), $r = 5$ (bottom left) and $r = 10$ (bottom right).} 
\label{numwave}
\end{center}
\end{figure}

Interestingly, these results confirm the observations on synthetic data sets: 
\begin{itemize}

\item For the CBCL data set (a nonnegative matrix), A3 identifies an optimal solution while the other approaches are not able to (although performing more iterations of Algorithm~\ref{cdsemi} would improve their solutions; see Figure~\ref{numcbcl}). 

\item For the Ionosphere data set, when $r$ is small ($r = 3$), A3 is not able to identify a good initial point and perform the worse (as for Gaussian matrices with $r = 20$; see Figure~\ref{numexp3}). 
When $r$ is large ($r = 10$), A3 is again the only approach that leads to an optimal solution matching the error of the best rank-$r$ approximation. For $r = 5$, it does not perform best, but allows to obtain a rather good initial point (A3/10 performs best). 

\item For the Wave data set, A3 is always able to identify an optimal solution, because its best rank-$r$ approximation is semi-nonnegative (recall that if it is semi-nonnegative for some $r$, it is for all $r' \geq r$; see Section~\ref{snm}). 

\item A3 allows Algorithm~\ref{cdsemi} to converge very quickly, in all cases in less than 10 iterations. 

\end{itemize}

It is interesting to note that, for these experiments, RD performs better than KM although the difference is not significant (except for the CBCL face data set).

\section{Conclusion} 

In this paper, we have addressed theoretical questions related to semi-NMF that led us to the design of exact and heuristic algorithms. 
Our contribution is three-fold. We showed that 
\begin{itemize}
\item The approximation error of semi-NMF of rank $r$ has to be smaller than the approximation error of its unconstrained counterpart of rank $r-1$. This result allowed us to design a new initialization procedure for semi-NMF that guarantees the error to be equal to the error of the best rank-$(r-1)$ approximation; see Theorem~\ref{th1} Algorithm~\ref{svdinit}.  
However, it seems that this initialization procedure does not work very well in practice. 

\item Exact semi-NMF can be solved in polynomial time (Theorem~\ref{th3}), and semi-NMF of a matrix $M$ can be solved in polynomial time up to any given precision with Algorithm~\ref{exactseminmf} given that the best rank-$r$ approximation of $M$ is semi-nonnegative. 
Algorithm~\ref{exactseminmf} can also handle cases when the aforementioned condition is not met, and we illustrated its effectiveness on several synthetic data sets. 

\item Semi-NMF is NP-hard in general, already in the rank-one case (Theorem~\ref{th5}). Moreover, we showed that some semi-NMF instances are ill-posed (that is, an optimal solution does not exist).  

\end{itemize} 

Further research on semi-NMF includes the design of other initialization strategies 
(in particular in the case $r$ is small and the best rank-$r$ approximation of $M$ is far from being semi-nonnegative; see also Remark~\ref{remeps}),  and the analysis of constrained variants of semi-NMF such as sparse semi-NMF where $V$ is required to be sparse. 
In fact, sparse semi-NMF would in general make more sense as it has better clustering properties; 
see the discussions in Section~\ref{secnm} and in \cite{DTJ10} for more details.   
In particular, it would be interesting to see how Algorithm~\ref{exactseminmf} performs as an initialization strategy in that case.

 \section*{Acknowledgments}
 
The authors would like to thank the reviewers for their insightful comments which helped improve the paper.

\small 

\bibliographystyle{spmpsci}
\bibliography{Biography}

\end{document}